\documentclass[11pt]{article}
\usepackage{amsmath} %[intlimits]

\usepackage{amssymb}
\usepackage{amsfonts}
\usepackage{latexsym}
\usepackage{amsthm}
\usepackage{amsxtra}
\usepackage{amscd}
\usepackage{bbm}
\usepackage{mathrsfs}
\usepackage{bm} 
\usepackage{dsfont} 
\usepackage{stmaryrd}
\usepackage{cite}
\usepackage{graphicx, color}
\usepackage{xspace}
\usepackage{authblk}

\usepackage{color}

\numberwithin{equation}{section}

\newcommand{\pt}{\partial}

\newcommand{\dd}{\mathrm{d}}

\newcommand{\al}{\alpha}
\newcommand{\de}{\delta}
\newcommand{\De}{\Delta}
\newcommand{\e}{{\varepsilon}}
\newcommand{\eps}{\epsilon}
\newcommand{\ga}{{\gamma}}
\newcommand{\Ga}{{\Gamma}}
\newcommand{\la}{\lambda}
\newcommand{\La}{\Lambda}

\newcommand{\om}{{\omega}}

\newcommand{\ol}[1]{\overline{#1}} %overline

\newcommand{\be}{\begin{equation}}
\newcommand{\ee}{\end{equation}}
\newcommand{\ben}{\begin{equation}\nonumber}

\newcommand{\floor}[1] {\lfloor {#1} \rfloor}

\newcommand{\ind}[1]{\mathbbm{1}_{#1}}
\newcommand{\pr}[1]{\left(#1\right)}

\renewcommand{\P}[1]{\mathbb{P}\left[#1\right]}
\newcommand{\E}[1]{\mathbb{E}\left[#1\right]}
\newcommand{\R}{\mathbb{R}}

\newcommand{\C}{\mathbb{C}}
\newcommand{\N}{\mathbb{N}}
\newcommand{\Z}{\mathbb{Z}}

\newcommand{\U}{\mathbb{U}}
\newcommand{\F}{\mathcal{F}}

\newcommand{\absa}[1]{\left\lvert #1 \right\rvert}
\newcommand{\norma}[1]{\left\lVert #1 \right\rVert}
\newcommand{\norm}[1]{\lVert #1 \rVert}

\DeclareMathOperator{\tr}{tr}
\DeclareMathOperator{\var}{Var}

\DeclareMathOperator{\im}{Im}
\DeclareMathOperator{\re}{Re}

\theoremstyle{plain} %plain, definition, remark
\newtheorem{theorem}{Theorem}[section]
\newtheorem*{theorem*}{Theorem}
\newtheorem{lemma}[theorem]{Lemma}
\newtheorem*{lemma*}{Lemma}
\newtheorem{corollary}[theorem]{Corollary}
\newtheorem*{corollary*}{Corollary}

\newtheorem*{proposition*}{Proposition}
\newtheorem{definition}[theorem]{Definition}
\newtheorem*{definition*}{Definition}

\newtheorem*{example*}{Example}
\theoremstyle{remark} %plain, definition, remark
\newtheorem{remark}[theorem]{Remark}

\newtheorem*{remark*}{Remark}
\newtheorem*{remarks*}{Remarks}
\newtheorem{notation}[theorem]{Notation}
\newtheorem*{notation*}{Notation}

\newcommand{\T}{{\mathbb T}}
\newcommand{\I}{\mathbb I}
\newcommand{\J}{\mathbb J}
\newcommand{\K}{\mathbb K}
\renewcommand{\S}{\mathbb S}
\newcommand{\M}{\mathcal M}
\newcommand{\Mp}{\mathcal M^+}
\newcommand{\Om}[1]{\mathcal O_{\!\sigma\!,p}\!\left(#1\right)}
\renewcommand{\O}[1]{\mathcal O\left(#1\right)}
\newcommand{\ii}{\mathrm{i}}

\newcommand{\cT}{\mathcal T}
\newcommand{\cK}{\mathcal K}
\newcommand{\tm}{\tilde m}

\newcommand{\bma}{\begin{bmatrix}}
\newcommand{\ema}{\end{bmatrix}}
\newcommand{\wo}[1]{^{(#1)}}
\newcommand{\D}{\mathcal{D}}

\title{Universality of random matrices with correlated entries}

\author{Ziliang Che\thanks{The work of the author is partially supported by the NSF grant DMS-1307444.}}
\affil{Harvard University\\ zche@math.harvard.edu}

\date{}
\begin{document}
\maketitle
\abstract{We consider an $N$ by $N$ real symmetric random matrix $X=(x_{ij})$ where $\E{x_{ij}x_{kl}}=\xi_{ijkl}$.  
Under the assumption that $(\xi_{ijkl})$ is the discretization of a piecewise Lipschitz function and that the correlation is short-ranged we prove that the empirical spectral measure of $X$ converges to a probability measure.
%We prove that as $N\to\infty$, the normalized empirical measure of eigenvalues of $X$ converges to a probability measure, assuming that $(\xi_{ijkl})$ is the discretization of a piecewise Liptchitz function and that the correlation is short-ranged.
  The Stieltjes transform of the limiting measure can be obtained by solving a functional equation.   Under the slightly stronger assumption that $(x_{ij})$ has a strictly positive definite covariance matrix, we prove a local law for the empirical measure down to the optimal scale $\im z \gtrsim N^{-1}$.
%Under a slightly stronger assumption that $(x_{ij})$ has a strictly positive definite covariance matrix, we prove a local law for the Stieltjes transform of the empirical measure, down to the optimal scale $\im z \gg N^{-1}$.  
The local law implies delocalization of eigenvectors.  As another consequence we prove that the eigenvalue statistics in the bulk agrees with that of the GOE. }
%\tableofcontents
\section{Introduction}
The Wigner-Dyson-Mehta conjecture asserts that the local eigenvalue statistics of large random matrices are universal in the sense that they depend only on the symmetry class of the model - real symmetric or complex Hermitian - but are otherwise independent of the underlying details of the model, such as the distribution of the individual matrix entries.  In particular they agree with the case that the entries are real or complex iid Gaussians - the so-called Gaussian Orthogonal and Unitary Ensembles (GOE/GUE) - for which there are explicit formulas.  The past decade has seen spectacular progress in the study of local statistics of random matrix ensembles.  In a series of works \cite{Erdos2009b, Erdos2012a, Erdos2010universality, Erdos2012b, Erdos2013a, Erdos2012c} the Wigner-Dyson-Mehta conjecture was established for Wigner ensembles which consist of random matrices with independent entries of identical variance.  Parallel results were obtained independently in various cases in \cite{Tao2011a, Tao2010}.  

The Wigner-Dyson-Mehta conjecture extends beyond the class of Wigner ensembles.  In fact the results of \cite{Erdos2009b, Erdos2012a, Erdos2010universality, Erdos2012b, Erdos2013a, Erdos2012c} also apply to generalized Wigner ensembles in which the variances are allowed to vary but are assumed to be of the same order, and the matrix of variances is assumed to be stochastic.  In the work  \cite{Ajanki2015} the authors consider matrix ensembles of general Wigner-type in which the stochasticity condition on the variances is dropped.  In \cite{ziliang} the authors consider the adjacency matrix of a sparse random graph model whose degree distribution satisfies a power law.  For such models the global statistics no longer follow the semicircle law, however the universality of the local statistics is unchanged.   Universality has been proved for a class of random band matrices \cite{Yau16}, deformed Wigner ensembles \cite{Kevin1,Kevin2} and the adjacency matrices of sparse random graphs \cite{Erdos2012a, Landon2015, ziliang}.

The study of these models has relied heavily on the independence of the matrix entries and the local statistics of matrices with a general correlation structure have not been considered.  Existing work on correlated random matrices has been restricted to models that have a specific correlation structure which could be exploited.  The universality of the adjacency matrices of  random regular graphs was obtained in \cite{Bau15,Huang15}.  Universality for Gaussian matrices with a translation invariant correlation structure was studied in \cite{Alt2015,Erdos2015}.  The local law was obtained for certain additive models in \cite{Bao15a,Bao15b}.  Universality was obtained for sparse random graph Laplacians in \cite{hl} and for polynomials of certain Gaussian matrices in \cite{Guionnet16,Guionnet15}.  Apart from results on the local scale, the convergence of empirical measure on the global scale was obtained for models with translation invariant correlation (see e.g. \cite{Banna13,Peligrad2016,Pastu94,Boutet99,Boutet96}) and for a model with piecewise translation invariant correlation \cite{Anderson2008}.  In all cases the analysis relied on the special structure of the matrix ensemble.

In this article we extend the Wigner-Dyson-Mehta conjecture to real symmetric random matrices with a general correlation structure, $\E{x_{ij}x_{kl} }= \xi_{ijkl}$.  Our assumptions are that $\xi$ is the discretization of piecewise Lipschitz function, and that the correlation is short-ranged, i.e., that the $(i, j)$th and $(k, l)$th entry are independent if either $|i-k| $ or $|j-l| >K$ where $K$ is fixed.  Under these conditions we obtain a global law for the empirical eigenvalue density.  Under the additional hypothesis that $\xi$ is strictly positive definite we obtain a local law as well as universality - that the local statistics coincide with the GOE in the limit $N \to \infty$.  

In addition to proving universality for Wigner matrices, the works of Erd\H{o}s-Yau et.al, \cite{Erdos2009b, Erdos2012a, Erdos2010universality, Erdos2012b, Erdos2013a, Erdos2012c} established a robust framework for proving universality for general matrix models.  This approach consists of a three-step strategy:
\begin{enumerate}
\item Obtain a local law, or high probability estimate on the empirical eigenvalue density at short scales.
\item Analyze the convergence of Dyson Brownian motion to local equilibrium.
\item A perturbation argument proving that the statistics remain unchanged by the Dyson Brownian motion flow.
\end{enumerate}

The general strategy of proving the local law consists of analyzing the Green's function
\ben
G(z) := \frac{1}{ H-z}
\ee
together with its normalized trace $m_N(z) = N^{-1} \tr (G)$.  The strategy developed in \cite{Erdos2012b,wegner,short,complete} for proving the local law for Wigner matrices consists of two key ingredients.  The first is a concentration estimate which implies that $m_N(z)$ satisfies an approximate fixed point equation with high probability:
\be\label{eqn:approxfixed}
m_N (z) = F ( m_N (z) ) + o(1).
\ee
The second ingredient is that the fixed point equation $m=F(m)$ is {\it stable}, i.e., that an approximate solution is in fact close to the solution.  In the case of the semicircle law and Wigner matrices the stability of the above scalar equation is trivial.  In order to prove the local law for generalized Wigner matrices one adopts a similar strategy but in this case shows that the vector $v:=( G_{ii} (z) )_i$ satisfies an approximate vector fixed point equation.  In this case the matrix of variances is stochastic and so the solution is in fact a constant vector.

In order to study the general Wigner type matrices\cite{Ajanki2015}, the Gaussian matrix with translation-invariant correlation \cite{Erdos2015} and the sparse random graph ensembles of \cite{ziliang} one shows, as in the generalized Wigner case, that the vector $v=(G_{ii} (z) )_i$ satisfies a vector fixed point equation
\be\label{eq:qve}
v_i (z) = \frac{1}{ - z - \sum_{ j} S_{ij} v_j(z) } + o(1).
\ee
When the matrix of variances $S$ is not stochastic the solution to the above fixed point equation is not in general a constant vector.  This type of equations is often crucial in identifying the limiting eigenvalue distributions of random matrices. One of the key contributions of \cite{Ajanki2015a}, \cite{ziliang} is to show that the above equation is stable in the bulk of the limiting spectrum, which is needed for proving the local law.

An important element in previous works on the local law is the independence of the matrix entries.  In particular, the $i$th row and column are independent of the $i$th minor; this key fact allows one to establish the approximate fixed point equation, and in the model considered here the loss of independence presents a serious challenge in this step of the proof.  We take advantage of the short-range nature of the correlation and  find that the entire matrix of Green's function elements $G = (G_{ij} )_{ij}$ satisfies a fixed point equation
\be \label{eqn:matfixed}
G = F (G) + o (1)
\ee
where now $F$ is a function on the space of $N \times N$ matrices.  In particular it is no longer sufficient to control only the trace $m_N$ or the vector of diagonal entries $(G_{ii})_i$, as the off-diagonal entries $G_{ij}$ are not necessarily small.   This generalizes the equations considered in, e.g. \cite{Ajanki2015a}, \cite{ziliang}.  The assumption that $\xi$ is a discretization of a Lipschitz function $\psi$ allows us to construct a limiting version of the equation \eqref{eqn:matfixed} on an auxilliary function space and establish stability for the finite $N$ equation.  A similar equation was derived in \cite{Pastur11}, where the authors considered $H_0+W$ where $H_0$ is a given symmetric matrix and $W$ is a Gaussian matrix with correlation between the matrix entries.  In \cite{Girko}, similar equations were considered for random matrices with certain block structures. The equation was studied in \cite{Helton2007} in a more general setting, where unique solvability was proved for an operator-valued self-consistent equation by applying the Earle-Hamilton fixed point theorem to a subdomain of a $C^*$-algebra.  We use the same argument to show the unique solvability and stability of the matrix equation \eqref{eqn:matfixed} for fixed $N$ and $z\in\C^+$.
	
In the limit $N\to \infty$, we obtain a limiting equation in a functional space $\cT:=L^\infty([0,1),\cK)$ where $\cK$ is the space of convolution operators on $l^2(\Z)$: \be\label{eqn:limit}
	m(\theta)=(-z-\Psi(m)(\theta))^{-1}\,,\forall \theta\in[0,1).\ee Here $\Psi$ is an integral operator and the inverse is taken in the space $\cK$.  The limiting equation \eqref{eqn:limit}, after Fourier transform, is a quadratic vector equation similar to \eqref{eq:qve}.  In \cite{Ajanki2015a}, it is proved that this type of quadratic vector equations are stable in the bulk of spectrum, which is needed in our proof of the local law.  The proof of stability in the bulk relies on the Krein-Rutman theorem on positive integral operators.  We then show that one can approximate the solution of the finite $N$ equation \eqref{eqn:matfixed} by the solution of \eqref{eqn:limit} with an $\O{N^{-1}}$ error. This approximation scheme enables us to prove the stability of \eqref{eqn:matfixed} in the bulk.

As mentioned above, the remainder of the strategy developed in  \cite{Erdos2009b, Erdos2012a, Erdos2010universality, Erdos2012b, Erdos2013a, Erdos2012c} to prove universality of consists of analyzing the DBM flow starting from our correlated matrix ensemble and showing that the statistics are unchanged under the DBM flow. %  The first is to analyze the DBM flow starting from our correlated matrix ensemble and show that the statistics converge to those of the GOE/GUE after a short time.  The final step is a perturbation argument relying on the matrix structure which proves that the statistics are unchanged under the DBM flow.

We modify the DBM flow in order to preserve the correlation structure of our model.  We then apply the results of \cite{Landon2015a} to show that the statistics agree with the GOE after a short time.  In the perturbation step we rely on the argument of \cite{Bourgade} which shows that the statistics are unchanged.

The main new ideas of this article are as follows. 1) In the case of matrices with independent entries, the Schur complement formula is a useful tool for "decoupling'' a row and column of the matrix from its minor, but is no longer effective in the correlated case.  We replace the Schur complement formula with a simple method that provides such a decoupling in the case of finite range correlations. %For  matrices with correlated entries, there is no simple decoupling idea  which can be used in conjunction with the Schur decomposition formula. We have found a simple  effective estimate related to correlations of Green functions and matrix elements based on resolvent identities. 
2) The off-diagonal entries of $G$ are possibly order of $1$, which does not happen for Wigner matrices and causes  serious trouble in the correlated case.  We solve this by finding a precise estimate for the off-diagonal entries based on the properties of solution of \eqref{eqn:matfixed}.  3) We develop a scheme to approximate the continuum solution of \eqref{eqn:limit} by the matrix solution of \eqref{eqn:matfixed}, which enables us to identify the limit of the Stieltjes transform $\tfrac{1}{N}\tr G$ and prove stability in the bulk.

In this article we focus on real symmetric matrices. The main results are easily generalized to complex Hermitian matrices.  Parallel results can be obtained for sample covariance matrices using the same approach, as this case can be reduced to analyzing the eigenvalues of Hermitian matrices. The assumption that the correlation is finite-ranged can be relaxed to, e.g. assuming that the correlation between the matrix entries decays exponentially fast (plus some more assumptions on weak dependence of distant entries).  However, this extension is technical and is not included in this article. %But this requires some technicality and will not be included in this article.  
 
We outline the rest of the article. In Section \ref{sec: definition} we define the model and lay out the assumptions, introduce the self-consistent equations, then state the main results.    In Section \ref{sec: derive sce}, we show that the Green's function satisfies the self-consistent equation up to an error term.   In Section \ref{sec: sce} we solve the self-consistent equation and prove that it converges to a limit in a certain sense; we also show that the self-consistent equation is stable under small perturbation.  Section \ref{sec: law} is devoted to proving two of the main results,  the global law and the local law of the Stieltjes transform of empirical measure.  In Section \ref{sec: bulk universality}, we prove the universality of local statistics of eigenvalues using the local law and a result in \cite{Landon2015a}.
%%%%%%%%%%%%%%%%%%%%%%%%%%%%% definition main results
%%%%%%%%%%%%%%%%%%%%%%%%%%%%% definition main results
%%%%%%%%%%%%%%%%%%%%%%%%%%%%% definition main results
\section{Definition and main results}\label{sec: definition}
\subsection{Definition of the model}
For each $N\in\N$, we consider an array of centered real random variables $(x_{ij})_{1\leq i\leq j\leq N}$.  We assume that  there is a four dimensional tensor $\xi=\xi\wo{N}$ such that
\be\label{def: xi}
	\E{ x_{ij}x_{kl}}=\xi_{ijkl}\,.
\ee
 We assume that the $(x_{ij})$ are $K$-dependent for some constant $K>0$ in the following sense:
\begin{definition}
A sequence of random variables  $(a_i)$ is $K$-dependent if $a_i$ is independent of $(a_j)_{\absa{j-i}>K}$.  A family of random variables $(a_{ij})$ is $K$-dependent if $a_{ij}$ is independent of $(a_{kl})_{\absa{i-k}\vee\absa{j-l}>K}$.
\end{definition}

In this article, a constant only depending on $K$ and $(\mu_p)$ (defined below) is regarded as a universal constant and we will omit the dependence.  We would like the model to include the adjacency matrices of random sparse graphs, in which there are roughly $N^\tau$ edges connected to each of $N$ vertices.  We  therefore introduce a sparsity parameter $q=N^\tau$ for some fixed $\tau\in(0,1]$.  We assume that there is a sequence of constants $(\mu_p)_{p\in\N}$ such that  $(x_{ij})$ satisfies the bounds
\be\label{eq: moments}
	\sup_{i,j}\E{\absa{x_{ij}}^p}\leq (N/q)^{p/2-1}\mu_p^p\,.
\ee
Without loss of generality assume $\mu_2=1$ so that $\sup_{i,j}\var{x_{ij}}\leq 1$. Note that when $\tau<1$, the $p$-th moments of $x_{ij}$ ($p>2$) are going to infinity as $N\to\infty$.

Now  consider a symmetric matrix $X$  whose upper-triangular part is $(x_{ij})_{1\leq i\leq j\leq N}$.   Since we are interested in the asymptotic behavior of the spectrum of $X$ as $N\to\infty$, we normalize $X$ by $N^{-1/2}$:
\ben
	H=\frac{1}{\sqrt{N}}X\,,
\ee
so that $\norma{H}$ is roughly of order 1.  We are going to analyze the Green's function $G(z)$ given by
\ben
	G(z)=(H-z)^{-1}\,, z\in\C^+=\{\zeta\in\C:\im\zeta>0\}\,.
\ee
Throughout this article, we always denote the imaginary part of $z$ by $\eta$.  The empirical measure $\mu_N:=\tfrac{1}{N} \sum_{i}\de_{\lambda_i(H)}$ of $H$  satisfies
\ben
	\int_\R \frac{\dd\mu_N( x)}{x-z}= \tfrac{1}{N}\tr G(z)\,.
\ee
 As $N\to \infty$, the limit of the empirical measure, if it exists, is in general not the semicircle law unless the correlation between matrix entries are rather weak.  
 
 In order to get a meaningful limit, one needs some mild  assumptions on $\xi$.  Let $\psi: [0,1)^2\times\Z^2\to\R$ be piecewise Lipschitz in the sense that there is a partition of $[0,1)$ into finitely many disjoint intervals
\be\label{def: partition}
	[0,1)=\cup_{\alpha \in\mathcal A} I_\alpha\,,
\ee such that $\psi(\cdot,\cdot,k,l)$ is Liptchitz on $I_\alpha\times I_{\alpha'}$ for any $\alpha,\alpha'\in\mathcal A$  and $k,l$.  Assume that $\xi$ is the discretization of $\psi$, i.e., for $i\leq j$, $k\leq l$,
 \ben
 	\xi\wo{N}_{ijkl}=\psi(i/N,j/N,k-i,l-j)+\O{N^{-1}}\,,
\ee
unless $(i/N,j/N)$ is in the $K/N$-neighborhood of a discontinuous point of $\psi(\cdot,\cdot,k-i,l-j)$.  This generalizes the model in \cite{Anderson2008} where $\psi$ is a step function on $[0,1)^2\times\Z^2$.  In order to be compatible with the symmetric structure of $X$, we assume $\psi(\theta,\phi,k,l)=\psi(\theta,\phi,-k,-l)=\psi(\phi,\theta,l,k)$ for all $(\theta,\phi,k,l)$.

Finally, we state a condition that will be assumed in some but not all of the main results.  In each main result, we will specify whether the condition is assumed or not.
\begin{definition}\label{def: pd}
	Let $\Sigma\wo{N}$ be the covariance matrix of the family $(x_{ij})_{1\leq i\leq j\leq N}$.  The tensor $\xi$ is said to be \textbf{positive definite} with lower bound $c_0>0$ if $\Sigma\wo{N}\geq c_0$ for all $N$.
\end{definition}
When the condition is assumed, any constant that depends on on $c_0$ will be seen as a universal constant and we will omit the dependence. Here we remark that the condition is very mild, since if $(w_{ij})$ is a family of i.i.d. random variables with mean $0$ and a small variance $\e>0$, then the $\xi$ of $(x_{ij}+w_{ij})$ will be positive definite with lower bound $\e$.

\begin{remark} The assumption that $K$ is fixed can be relaxed to e.g. $K=(\log N)^{\log \log N}$ with little technical difficulty. We refrain from doing so in order to make the argument transparent to the reader.
\end{remark}
\begin{remark}
In this article we focus on real symmetric matrices. However, as mentioned in the introduction, the argument can be easily generalized to the complex case.  
\end{remark}

\subsection{Self-consistent equations}\label{subsec: sce}
Before stating the main results, we would like to introduce the equations that the Green's function satisfies. As $\xi_{ijkl}$ has only been defined for $i\leq j$, $k\leq l$, we will extend it to all $(i,j,k,l)\in\N^4$.  For technical reasons we extend it in such a way: for $i>j$ or $k>l$ define 
\ben
\xi_{ijkl} = \begin{cases} \xi_{jilk} &\mbox{if } i>j\,, k>l\\ 
0& \mbox{otherwise} \end{cases}.
\ee
Note that now \eqref{def: xi} is \textbf{not} always true for each $(i,j,k,l)$.   To compensate for this, let $w$ be a random variable uniformly distributed on $[0,1)$ and define 
\be\label{def: hat x}
	\hat x_{ij} = x_{ij} e^{\ii 2\pi w} \text{ if } i\leq j\,; \hat x_{ij}=x_{ij},\text{ otherwise. }
\ee
One can easily see that $\xi_{ijkl}=\E{\hat x_{ij}^*\hat x_{kl}}$ holds for all $(i,j,k,l)\in\N^4$.  For each $N$ define a matrix-valued map $\Xi:\C^{N\times N} \to \C^{N\times N}$ through
\be\label{def: Xi}
	(\Xi(M))_{ik}:=\tfrac{1}{N}\sum_{j,l}\xi_{ijkl}M_{jl} \,,\forall i,k\in\N_N\,.
\ee
For each  $N\in\N$ and $z\in\C^+$ we consider the equation
\be\label{eq: sce n}
	M(-z-\Xi(M))=I\,.
\ee
We will show that the Green's function $G$ is approximately the solution of this equation when $N$ is large. In Section \ref{sec: sce} we show that the equation is uniquely solvable in a certain class of matrices, by a fixed point argument.

Naturally, one expects $M$ has some sort of limit when $N\to\infty$.  In order to identify the limit, we consider the space 
\be\label{def: cT}
	\cT=L^\infty([0,1),\cK)
\ee
where $\cK$ is the space of bi-infinite sequences that are viewed as convolution operators on $l^2(\Z)$, i.e., \[\cK=\{(a(k))_{k\in\Z}: \norma{a*f}_{l^2(\Z)}<\infty\,,\text{ for all } f\in l^2(\Z)\}.\]  An element in $\cT$ can be regarded either as a function from $[0,1)\times \Z\to\C$ or a function from $[0,1)\to \cK$.  To clarify notations, for any $f\in\cT$ and  each $\zeta\in [0,1)$, we denote $f(\zeta):=(f(\zeta,k))_{k\in\Z}$ as a member in $\cK$.  Let $\check f$ denote the inverse Fourier transform in the $k$ variable, i.e., $\check f(\theta,\zeta):=\sum_k f(\theta,k)e^{\ii 2\pi k\zeta}$. It is well known that the norm of $a=(a(k))_{k\in\Z}\in \cK$ satisfies $\norma{a}=\norma{\check a}_\infty$.  The norm on $\cT$ therefore satisfies
\be
	\norma{f}_\cT = \norma{\check f}_\infty\,.
\ee
The limiting version of trace is defined as follows:
\begin{definition}\label{def: tr infty}
For any $f\in\cT$, $\tr f :=\int_0^1 f(\theta, 0)\dd \theta$.
\end{definition}
The limiting version of the map $\Xi$ is an operator $\Psi$ defined through
\be\label{def: Psi}
	\Psi(f)(\theta,k)= \iint \psi(\theta,\phi,k,l)f(\phi,l)\dd \theta\dd l\,.
\ee
Here $\dd l$ denotes the counting measure on $\Z$. Define the inverse $f^{-1}$ of $f$ on $\cT$ through 
\be\label{def: inverse T}
	f^{-1}(\theta)=(f(\theta))^{-1}\,,\forall \theta\in[0,1)\,.
\ee
Here the second inverse is taken in the space $\cK$.  We regard $z \in \mathbb{C}$ as an element of $ f_z \in \cT$ given by $f_z ( \theta, k) = z \de_{k0}$. 
Now we are ready to write down the self-consistent equation on $\cT$:
\be\label{eq: sce infty}
	m=(-z-\Psi(m))^{-1}\,.
\ee
This equation is uniquely solvable in a certain subdomain of $\cT$.  It turns out that if $m$ is the solution to the equation above, $\tr m(z)$ (see Definition \ref{def: tr infty}) is the Stieltjes transform of a probability measure on $\R$, i.e., there is a probability measure $\mu$ on $\R$ such that 
\ben
	\tr m(z) =\int \frac{\mu(\dd x)}{x-z}\,, z\in\C^+.
\ee
If $\xi$ is positive definite in the sense of Definition \ref{def: pd}, then $\mu$ has a continuous density 
\be\label{def: rho}
	\mu(\dd x)=\rho(x)\dd x\,.
\ee
In Section \ref{sec: sce} we will see that under inverse Fourier transform, equation \eqref{eq: sce infty} becomes a quadratic vector equation as studied in \cite{Ajanki2015a,Ajanki2015b}, where the behavior of the limit density $\rho$ is described in detail.

\subsection{Main results}
Our first main theorem concerns the unique solvability of equation \eqref{eq: sce n} and equation \eqref{eq: sce infty}. The theorem follows from Theorem \ref{thm: solvable}, Theorem \ref{thm: solvable infty} and Theorem \ref{thm: converge local}.  Apart from unique solvability, the equations are also stable under small perturbations, but we will state the stability results in Section \ref{sec: sce}.  Recall the definition \eqref{def: cT} of the space $\cT$.  The operator norm of a matrix $A$ is denoted by $\norm{A}$.
\begin{theorem}
For any $N\in\N$ and $z\in\C^+$, the self-consistent equation \eqref{eq: sce n} has a unique solution $M=M(N,z)$ in the set $\{A\in\C^{N\times N}: \tfrac{1}{2\ii}(A-A^*)>0\}$.  For any $z\in\C^+$, equation \eqref{eq: sce infty} has a unique solution $m=m(z)$ in the set $\{f\in \cT: \inf_{(\theta,s)\in[0,1)^2} \check f(\theta,s)>0\}$.  

Let $\hat M$ be the discretization of $m$ defined through $\hat M_{i,i+k}:= m(i/N,k)$, then
\ben
	\norm{M-F(\hat M)}\leq c_z/N\,.
\ee
Moreover, assume that $\xi$ is positive definite in the sense of Definition \ref{def: pd}, let $\D\subset \C^+$ be a bounded domain such that $\im \tr m$ is bounded below, then there is a $c_\D$ such that 
\ben
	\norm{M-F(\hat M)}\leq c_\D/N\,,
\ee
uniformly for all $z\in\D$.
\end{theorem}
\begin{remark}
The assumption that $\im \tr m$ is bounded below on $\D$ might look a bit odd.  However, this is satisfied if $\rho(\re z)$ is bounded below for $z\in\D$, since $\rho(\re z)= \lim_{\im z\to 0^+}\tfrac{1}{\pi} \im \tr m(z)$.

\end{remark}

Before stating the limiting laws of $\tfrac{1}{N}\tr G$, we introduce a deterministic control parameter that will frequently appear throughout the article.
\begin{definition}
\ben
	\Phi=\Phi(N,z)=\frac{1}{\sqrt{N\eta}}+\frac{1}{\sqrt{q}}\,.
\ee
\end{definition}
Now we state the global law of $\tfrac{1}{N}\tr G$ which  says that with very high probability, $\tfrac{1}{N}\tr G(z)$ converges to $\tr m(z)$ uniformly in compact subsets of $\C^+$.  Moreover, each entry of $G$ is well approximated by the corresponding entry of the deterministic matrix $M$ which solves equation \eqref{eq: sce n}.  Note that $M$ is a $N\times N$ deterministic matrix depending on $N$ and $z$.  
\begin{theorem}[The global law]\label{thm: global law}
Let $\D\subset \subset\C^+$.  Assume that $M=M(N,z)$ solves equation \eqref{eq: sce n} and $m$ solve equation \eqref{eq: sce infty}.  Then for arbitrary $\nu>0$, and $p$ large enough, the following estimates hold when $N\geq N_{\D,\nu,p}$.
\ben
	\P{ \sup_{i,j\in\N_N\,,z\in\D}{\absa{G_{ij}-M_{ij}}}\geq N^\nu \Phi } \leq N^{-\nu p}\,.
\ee
\ben
	\P{\sup_{z\in\D}\absa{\tfrac{1}{N}\tr G- \tr m }\geq N^\nu \Phi}\leq N^{-\nu p}\,.
\ee
\end{theorem}
\begin{remark}
This theorem can be applied to the empirical measure of singular values $\{\mu_k\}$ of a non-symmetric random matrix $X$ under similar assumptions, since one can symmetrize $X$ by defining 
\ben
	Y=\bma 0 &X^*\\X&0\ema\,,
\ee
whose nontrivial eigenvalues are $\{\pm\mu_k\}$.
\end{remark} 

The local law of $\tfrac{1}{N}\tr G$ says that if $\re z$ is in the bulk of the limit density,  then with very high probability, $\tfrac{1}{N}\tr G(z)$ converges to $\tr m(z)$.  Moreover, each entry of $G$ is well approximated by the corresponding entry of the deterministic matrix $M$ which solves equation \eqref{eq: sce n}, with error roughly of the size $\Phi$.  
\begin{theorem}[The local law]\label{thm: local law}
Assume that $\xi$ is positive definite in the sense of Definition \ref{def: pd}.   Assume that $M=M(N,z)$ solves equation \eqref{eq: sce n} and $m$ solves equation \eqref{eq: sce infty}.  Fix a bounded domain $\D\subset \C^+$ such that $\rho(\re(z))$ (see \eqref{def: rho}) is bounded below by $\om>0$ for any $z\in\D$. For any $\nu\in(0,1]$ define $\D_\nu\wo{N}:=\{z=E+\ii \eta \in \D: \eta>N^{-1+\nu}\}$.
Then for $\sigma$ small enough and $p$ large enough, the following estimates hold for all $N\geq N_{\om,\sigma,p}$
\ben
	\P{ \sup_{i,j\in\N_N\,,z\in\D_\nu}{\absa{G_{ij}-M_{ij}}}\geq N^\sigma \Phi } \leq N^{-\sigma p}\,.
\ee
\ben
	\P{\sup_{z\in\D_\nu}\absa{\tfrac{1}{N}\tr G- \tr m }\geq N^\sigma \Phi}\leq N^{-\sigma p}\,.
\ee
\end{theorem}
Delocalization of eigenvectors says that the eigenvectors corresponding to bulk eigenvalues are flat.  Proving the corollary from Theorem \ref{thm: local law} can be done by a routine argument (see e.g. \cite{Erdos2013a}).
\begin{corollary}\label{cor: regidity}[Delocalization of eigenvectors]
Assume that $\xi$ is positive definite in the sense of Definition \ref{def: pd}.   Assume that $M$ solves equation \eqref{eq: sce n} and $m$ solve equation \eqref{eq: sce infty}.  Let $\ga_k:= \inf \{\ga: \int_{-\infty}^\ga \rho (x)\dd x=k/N\}$ be the $k$-th classical location of the limiting density, and $u_k=(u_k(i))_i$ be the eigenvector associated with $\la_k$. Let $\om>0$ be a fixed number and $\sigma$ be an arbitrarily small number, then 
\ben
	 \sup_{k:\rho(\ga_k)\geq \om\,, i\in\N} \absa{u_k(i)} ^2\leq N^{-1+\sigma}
\ee
hold with probability $1-N^{-\sigma p}$ for $p$ large enough and $N\geq N_{\om,\sigma,p}$.
\end{corollary}

As a consequence of Theorem \ref{thm: local law}, we have the universality of $k$-point correlation functions in the bulk:
\begin{theorem}[Bulk universality]\label{thm: bulk universality}
Assume that $\xi$ is positive definite (see Definition \ref{def: pd}).  Let $E\in\R$ be in the bulk of $\rho$, that is, $\rho$ has a positive density in a neighborhood of $E$.  Let $O$ be a test function on $\R^k$.  Fix a parameter $b=N^{-1+c}$ for any $c>0$.  We have,
\ben
	\begin{split}
	\lim_{n\rightarrow \infty} \frac{1}{2b}\int_{E-b}^{E+b} \int_{\R^n} O(\alpha_1,\dots,\alpha_k) \left[ \frac{1}{\rho(E)^k}\rho^{(k)} \left( E'+\frac{\alpha_1}{N\rho(E)},\dots,E'+\frac{\alpha_k}{N\rho(E)}\right)\right.\\\left.
	-\frac{1}{(\rho_{sc}(E))^k} \rho_{GOE}^{(k)} \left( E''+\frac{\alpha_1}{\rho_{sc}^{(N)}(E)},\dots,E''+\frac{\alpha_k}{\rho_{sc}^{(n)}(E)}\right)\right]\dd \alpha_1\dots\dd\alpha_k\dd E' =0
	\end{split}
\ee
\end{theorem}
%%%%%%%%%%%%%%%%%%%%%%%%% Derivation of sce
%%%%%%%%%%%%%%%%%%%%%%%%% Derivation of sce
%%%%%%%%%%%%%%%%%%%%%%%%% Derivation of sce

\section{Derivation of the self-consistent equations}\label{sec: derive sce}
In this section we will show that the Green's function $G$ approximately satisfies the equation \eqref{eq: sce n}, up to an error term of size $\Phi$.  The main result of this section is Lemma \ref{lem: sce app}.  The main tools are some algebraic identities stated in Subsection \ref{subsec: res} and concentration inequalities of quadratic forms of weakly dependent random variables  in Subsection \ref{subsec: concentration}.
\subsection{Resolvent indenities}\label{subsec: res}
\begin{notation}
\begin{enumerate}
\item
For $\T\subset\N$, denote $H^{(\T)}= (H^{(\T)}_{ij}):=(H_{ij}\ind{i\notin\T}\ind{j\notin\T})$ and $G^{(\T)}:=(H^{(\T)}-z)^{-1}$.  In the case $\T=\{k\}$, we write $H^{(k)}$ instead of $H^{(\{k\})}$.  If $k\notin \T$, we write $\T k:=\T \cup\{k\}$.
\item
For $\T,\S\subset\N$, Let $A_{\T,\S}$ be the submatrix of $A$ whose indices are in $\T\times\S$.
\item
For $z\in\C$, denote $\eta=\im z$, $E=\re z$.
\end{enumerate}
\end{notation}
\begin{definition}
We will frequently use a stochastic control parameter
\ben
	\Ga=1\vee\max_{ij}\absa{G_{ij}}\,.
\ee
For technical reasons we also need the following stochastic control parameter
\be\label{def: ga}
	\ga=1\vee \sup_i \sup_{\I,\J} \norma{(G_{\I,\I}\wo{\J})^{-1}}\,,
\ee
where the second $\sup$ is taken over all $\I,\J\subset[i-2K,\dots i+2K]$ such that $\I\cap\J=\emptyset$. 
\end{definition}

We prove two resolvent identities that we will frequently use in the article. These identities were named the first decoupling identity \cite{Erdos2012} and second resolvent decoupling identity \cite{Yau2013}. Although their proofs are very simple; they play fundamental roles in the proof of local laws for random matrices.
\begin{lemma}[Resolvent identities]\label{lem: res}
Let $H$ be a Hermitian matrix, $\T,\I,\J\subset\N$. If $\I\cap\T=\J\cap\T=\emptyset$, then
\be\label{eq: res diag}
	G_{\I,\J}=G_{\I,\J}^{(\T)}+G_{\I,\T}(G_{\T,\T})^{-1}G_{\T,\J}\,,
\ee
\be\label{eq: res off}
	G_{\T,\J}=-G_{\T,\T}H_{\T,\T^c}G_{\T^c,\J}^{(\T)}\,.
\ee
\end{lemma}
\begin{proof}
The second equation follows from taking the $(\T,\J)$-block of the resolvent identity
\ben
	G-G\wo{\T}=G(H\wo{\T}-H)G\wo{\T}\,.
\ee
On the other hand, taking the $(\I,\J)$-block of the same identity one has
\ben
	G_{\I,\J}=G_{\I,\J}^{(\T)}-G_{\I,\T}H_{\T,\T^c}G_{\T^c,\J}\wo{\T}\,.
\ee
Combining this equation with \eqref{eq: res off} yields \eqref{eq: res diag}.
\end{proof}

The first resolvent identity \eqref{eq: res diag} immediately impliy the following corollary, which will be used many times in the rest of the section.
\begin{corollary}\label{cor: off bd}
Let $\I\subset [i-2K,\dots i+2K]$, $k,l\notin \I$, then 
\ben
	\absa{G_{kl}\wo\I}\leq 8K^2 \Ga^2\ga\,.
\ee
\end{corollary}

We will also need the Ward identity for the resolvent of Hermitian matrices.
\begin{lemma}[Ward Identity]\label{lem: ward}
	Let $A$ be an $m\times m$ Hermitian matrix and $z=E+\ii\eta\in\C^+$.  Let $B=(A-z)^{-1}$.  Then,
\ben
	\sum_k \absa{B_{ik}}^2= \im B_{ii}/\eta\,,\forall 1\leq i\leq m.
\ee
\end{lemma}

%%%%%%%%%%%%%%%%%%%%%%%%%% concentration
%%%%%%%%%%%%%%%%%%%%%%%%%% concentration
%%%%%%%%%%%%%%%%%%%%%%%%%% concentration

\subsection{Concentration inequalities}\label{subsec: concentration}

The following lemma is a corollary of Lemma A.1 in \cite{Erdos2013b}. 
\begin{lemma}\label{lem: ldp}
Let $(a_i)$ be a family of $K$-dependent random variables satisfying 
\ben
	\E{ \absa{a_i}^p }\leq \frac{ \mu_p^p}{N q^{p/2-1}}\,.
\ee
Let $(A_i)$, $(B_{ij})$ be deterministic families or random variables that are independent of $(a_i)$. Then for any $\sigma>0$ and $p$ large enough we have with probability at least $1-c_p N^{-\sigma p}$
\be\label{eq: ldp linear}
	\absa{\sum_i A_ia_i-\E{\sum_i A_i a_i}}\leq N^\sigma\pr{\frac{\sup_i\absa{A_i}}{\sqrt{q}}+\sqrt{\tfrac{1}{N}\sum_i\absa{A_i}^2}}\,,
\ee

\ben
	\absa{\sum_{i,j}B_{ij}a_ia_j-\E{\sum_{i, j} B_{ij}a_ia_j}}\leq N^\sigma\pr{\frac{\sup_{i\neq j}\absa{B_{ij}}}{\sqrt{q}}+\sqrt{\frac{1}{N^2}\sum_i\absa{B_{ij}}^2}}\,.
\ee
Here $c_p$ depends only on $p$.
\end{lemma}
\begin{proof}
For the first estimate, one only need to split the sum $\sum_i A_i a_i$ into at most $K$ parts, each part being the sum of independent random variables.  One can apply Lemma A.1 in \cite{Erdos2013b} to each part and get the first estimate.
For the second estimate, one can write
\ben
	\sum_{i,j}B_{ij}a_ia_j-\E{\sum_{i, j} B_{ij}a_ia_j}=\sum_{\absa{i-j}>K}B_{ij}a_ia_j+\left(\sum_{\absa{i-j}\leq K}B_{ij}a_ia_j-\E{\sum_{\absa{i-j}\leq K} B_{ij}a_ia_j}\right)\,.
\ee
The first sum on the right hand side can be split into at most $K^2$ parts and each part is estimated using Lemma A.1 in \cite{Erdos2013b}.  The second sum on the right hand side can be estimated using \eqref{eq: ldp linear}.
\end{proof}
\begin{corollary}\label{cor: ldp G}
Under the same condition as Lemma \ref{def: om}.  Assume further that $A_i=G_{ij}\wo{\T}$, $B_{ij}=G_{ij}\wo{\T}$ for some $\T \subset [k-2K,\dots k+2K]$ and $k\in\N$.  Recall the definition \eqref{def: ga} of $\ga$.  Then for any $\sigma>0$ and $p\geq 2$ we have with probability at least $1-c_p N^{-\sigma p}$
\ben
	\absa{\sum_i A_ia_i-\sum_i A_i \E{a_i}}\leq cN^\sigma\Phi\Ga^2\ga\,,
\ee
\ben
	\absa{\sum_{i\neq j} B_{ij}a_ia_j-\sum_{i\neq j} B_{ij}\E{a_ia_j}}\leq cN^\sigma\Phi\Ga^2\ga\,.
\ee
Here $c$ is an universal constant.
\end{corollary}
\begin{proof}
By Lemma \ref{lem: ldp} and the Ward identity, for any $\sigma>0$ and $p\geq 2$ we have with probability at least $1-c_p N^{-\sigma p}$
\ben
	\absa{\sum_i G_{ij}\wo{\T} a_i -\E{\sum_iG_{ij}\wo{\T} a_i}}\leq N^\sigma\Phi \left(\sup_i \absa{G_{ij}\wo{\T}}+\sqrt{ \im \absa{G_{jj}\wo{\T}}}\right)\,,
\ee
which is $\O{N^\sigma\Phi\Ga^2\ga}$ in view of Corollary \ref{cor: off bd}.  The second estimate holds similarly.

\end{proof}

%%%%%%%%%%%%%%%%%%%%%%%%  subsect expansion
%%%%%%%%%%%%%%%%%%%%%%%%
%%%%%%%%%%%%%%%%%%%%%%%%
\subsection{Expansion of the Green's function}
Throughout this section, we fix an arbitrarily small $\sigma>0$ and an integer $p\geq 100/\sigma$.
\begin{definition}\label{def: om}
Let $(a\wo{N})$ and $(b\wo{N})$ be two sequences of random variables.  We say that $a=\Om{b}$ if there are universal constants $c, s>0$ and a constant $c_p$ depending only on $p$ such that  $\absa{a}\leq c b$ holds with probability at least $1-c_p N^{-\sigma p+s}$.
\end{definition}

We start with the trivial identity
\be\label{eq: trivial}
	\sum_kG_{ik}H_{kj}-G_{ij}z=\de_{ij}\,.
\ee
Let $\T$ be the set of indices that are correlated with $j$. 
\begin{lemma}
\be\label{eq: first expansion}
	-G_{i\T}H_{\T,\T^c}G_{\T^c,\T^c}\wo{\T}H_{\T^c,j}-G_{ij}z=\de_{ij}+\Om{N^\sigma\Phi\Ga^2\ga}\,.
\ee
\end{lemma}

\begin{proof}
Since $H_{ik}=\Om{N^\sigma\Phi}$, one can drop a few terms from the sum $\sum_kG_{ik}H_{kj}$,
\ben
	G_{i,\T^c}H_{\T^c,j}-G_{ij}z=\de_{ij}+\Om{N^\sigma\Phi\Ga}\,.
\ee
There are two cases:
\begin{enumerate}
\item $i\in\T$. By \eqref{eq: res off} we have
\be\label{eq: first expansion 1}
	-G_{i\T}H_{\T,\T^c}G_{\T^c,\T^c}\wo{\T}H_{\T^c,j}-G_{ij}z=\de_{ij}+\Om{N^\sigma\Phi\Ga}\,.
\ee
\item $i\notin\T$. By \eqref{eq: res diag} we have
\ben
	G_{ik}\wo{\T}H_{kj}+G_{i,\T}(G_{\T,\T})^{-1} G_{\T,k}H_{kj}=\de_{ij}+\Om{N^\sigma\Phi\Ga}\,.
\ee
The first term on the left hand side is $\Om{N^\sigma\Phi \Ga^2\ga}$ by Corollary \ref{cor: ldp G}, while the second term equals $-G_{i,\T}H_{\T,\N}G_{\N,\N}\wo{\T}H_{\N,j}$ by \eqref{eq: res off}. Therefore,
\be\label{eq: first expansion 2}
	-G_{i\T}H_{\T,\T^c}G_{\T^c,\T^c}\wo{\T}H_{\T^c,j}-G_{ij}z=\de_{ij}+\Om{N^\sigma\Phi\Ga^2\ga}\,.
\ee
\end{enumerate}

The lemma follows from \eqref{eq: first expansion 1} and  \eqref{eq: first expansion 2}.
\end{proof}

Let $k\in \T$, we are going to estimate $H_{k,\T^c}G_{\T^c,\T^c}\wo{\T}H_{\T^c,j}\,$ that appears on the left hand side of \eqref{eq: first expansion}.
Now let $\S\subset\T^c$ be the set of indices correlated with $\T$ and let $\U=\S\cup\T$.  Then we can split
\be\label{eq: split}
	H_{k,\T^c}G_{\T^c,\T^c}\wo{\T}H_{\T^c,j}=H_{k,\U^c}G_{\U^c,\U^c}\wo{\T}H_{\U^c,j}+H_{k,\U^c}G_{\U^c,\S}\wo{\T} H_{\S,j}+H_{k,\S}G_{\S,\T^c}\wo{\T}H_{\T^c,j}\,.
\ee
By \eqref{eq: res diag}, $G_{\U^c,\U^c}\wo{\T}=G_{\U^c,\U^c}\wo{\U}+G_{\U^c,\S}\wo{\T}(G_{\S,\S}\wo{\T})^{-1} G_{\S,\U^c}\wo{\T}$. Then \eqref{eq: split} becomes simply
\be\label{eq: split 1}
	H_{k,\T^c}G_{\T^c,\T^c}\wo{\T}H_{\T^c,j}=H_{k,\U^c}G_{\U^c,\U^c}\wo{\U}H_{\U^c,j}+\Upsilon_{kj}\,,
\ee
where
\be\label{eq: split upsilon}
	\Upsilon_{kj}=H_{k,\U^c}G_{\U^c,\S}\wo{\T}(G_{\S,\S}\wo{\T})^{-1} G_{\S,\U^c}\wo{\T}H_{\U^c,j}+H_{k,\U^c}G_{\U^c,\S}\wo{\T} H_{\S,j}+H_{k,\S}G_{\S,\T^c}\wo{\T}H_{\T^c,j}\,.
\ee

The good news is that one can condition on the index set $\U^c$ and apply Lemma \ref{lem: ldp} to the first term on the right hand side of \eqref{eq: split 1}, which yields the following lemma:
\begin{lemma}\label{lem: sce app}
\ben
	-\sum_{k,l,m}G_{ik}\xi_{kljm}G_{lm}-G_{ij}z=\de_{ij}+\Om{N^{2\sigma}\Phi\Ga^5\ga^3}\,.
\ee

\end{lemma}
\begin{proof}
In view of \eqref{eq: first expansion} and \eqref{eq: split 1}, it is sufficient to prove that
\ben
	H_{k,\T^c}G_{\T^c,\T^c}\wo{\U}H_{\T^c,j}+\Upsilon_{kj}=\tfrac{1}{N}\sum_{l,m}\xi_{kljm}G_{lm}+\Om{N^{2\sigma}\Phi\Ga^4\ga^3}\,.
\ee
We first estimate the first term on the left hand side. By Corollary \ref{cor: ldp G} and the fact that $\E{x_{kl}x_{jm}}=\xi_{kljm}$ holds with at most $\O{1}$ exceptions for each $(k,j)$, we have
\ben
	H_{k,\U^c}G_{\U^c,\U^c}\wo{\U}H_{\U^c,j}=\tfrac{1}{N}\sum_{l,m}\xi_{kljm}G_{lm}\wo{\U} +\Om{N^\sigma\Phi\Ga^2\ga}\,.
\ee
By Corollary \ref{cor: off bd},
$
	\tfrac{1}{N}\sum_{l,m}\xi_{kljm}G_{lm}\wo{\U}=\tfrac{1}{N}\sum_{l,m}\xi_{kljm}G_{lm}+\O{\Phi^2\Ga\ga}\,.
$
Here we have used Cauchy-Schwarz inequality and Lemma \ref{lem: ward}. Plugging into the above equation we have
\be\label{eq: split 2}
	H_{k,\U^c}G_{\U^c,\U^c}\wo{\U}H_{\U^c,j}=\tfrac{1}{N}\sum_{l,m}\xi_{kljm}G_{lm} +\Om{N^\sigma\Phi\Ga^2\ga}\,.
\ee
Then we estimate $\Upsilon_{kj}$. By  \eqref{eq: res off} we have
\ben
	H_{\T,\U^c}G_{\U^c,\S}\wo{\T}=-(G_{\T,\T})^{-1} G_{\T,\S}+\Om{N^\sigma\Ga^2\ga}=\Om{N^\sigma\Ga^2\ga}\,.
\ee
On the other hand we use Lemma \ref{lem: ldp} to get 
\ben
	G_{\S,\U^c}\wo{\T}H_{\U^c,j}=\Om{N^\sigma\Phi\Ga^2\ga}\,,\quad G_{\S,\T^c}\wo{\T}H_{\T^c,j}=\Om{N^\sigma\Phi\Ga^2\ga}\,.
\ee
Combining all the three estimates above we get
\be\label{eq: split 3}
	\Upsilon_{kj}=\Om{N^{2\sigma}\Phi\Ga^4\ga^3}\,.
\ee
The lemma follows from \eqref{eq: split 2} and \eqref{eq: split 3}.
\end{proof}
%%%%%%%%%%%%%%%%%%%%%%%%  subsect sce
%%%%%%%%%%%%%%%%%%%%%%%%
%%%%%%%%%%%%%%%%%%%%%%%%
\section{Solving the self-consistent equations}\label{sec: sce}
In this section, we show the unique solvability and stability of equation \eqref{eq: sce n} and \eqref{eq: sce infty}. We also show that as $N$ goes to infinity, the solution $M$ to \eqref{eq: sce n} converges to the solution $m$ of \eqref{eq: sce infty} in a certain sense. We give an explicit construction of $M$ from $m$, up to an error of size $\O{N^{-1}}$.

%%%%%%%%%%%%%%%%%%%%%%%%  subsect sce n
%%%%%%%%%%%%%%%%%%%%%%%%
%%%%%%%%%%%%%%%%%%%%%%%%
\subsection{Solution for fixed $N$ and $z$}
The strategy to solve \eqref{eq: sce n} is to write it as a fixed point equation and apply the Earle-Hamilton fixed point theorem to a certain subdomain of $\C^{N\times N}$.  This was done by Helton \textit{et al.} \cite{Helton2007} in a general setting for operator-valued self-consistent equations.  For the readers' convenience, we give self-contained proofs in our case.  We also prove the stability of \eqref{eq: sce n} under perturbations that has small $\norm{\cdot}_\infty$ norm.  This relies on the off-diagonal decay of the solution.  The main results of this subsection are Theorem \ref{thm: solvable} and Theorem \ref{thm: stable 1}.

 We restate the self-consistent equation below, recalling the map $\Xi$ defined in \eqref{def: Xi}:
\ben
	M(-z-\Xi(M))=I\,,z\in\C^+\,.
\ee
It is remarkable that $\Xi(M)$ is a band matrix with band width $2K+1$, which will be used to prove the exponential decay of off-diagonal entries of the solution.  We are going to solve the equation using a fixed-point argument.  For this purpose, we define
\be\label{def: F}
	F(M):=(-z-\Xi(M))^{-1}\,.
\ee
Then equation \eqref{eq: sce n} becomes simply
\ben
	M=F(M)\,.
\ee
However, $F$ is not defined on the entire space $\C^{N\times N}$, therefore we introduce a domain where $F$ is well-defined.  

In the sequel, we denote the operator norm of a matrix $A$ by $\norm{A}$ and denote $\max_{i,j}\absa{A_{ij}}$ by $\norm{A}_\infty$.  Note that $\norm{\cdot}$ is a stronger norm since $\norm{A}_\infty\leq \norm{A}$.
\begin{definition}
Define $\Mp_N=\{M\in\C^{N\times N}: \tfrac{1}{2\ii}(M-M^*)>0\}$. For $\e,\delta\geq0$, define $\M_N({\e,\delta})= \{M\in\C^{N\times N}: \tfrac{1}{2\ii}(M-M^*)> \e\,, \norma{M}< \delta\}$.
\end{definition}
The following lemma says that the map $F$ is well defined on $\Mp_N$ and is from $\Mp_N$ to itself, since $F$ is the composite of two maps $M\mapsto z+\Xi(M)$ and $M\mapsto -M^{-1}$.
\begin{lemma}\label{lem: closed}
The space $\Mp_N$ is closed under addition and closed under the map
\ben
	M\mapsto z+\Xi(M)\,,
\ee
for any $z\in \C^+$.  

It is also closed under the map $M\mapsto -M^{-1}$, which maps $\M_N({\e,\de})$ into $\M_N({\frac{\e}{\de^2},\frac{1}{\e}})$.
\end{lemma}
\begin{proof}
Clearly $\Mp_N$ is closed under addition. Now suppose $M\in\Mp_N$, in order to show $z+\Xi(M)\in\Mp_N$, we need to show $\eta+\tfrac{1}{2\ii}\Xi(M-M^*)>0$.  Note that the map $\Xi$ has a representation 
\be\label{eq: xi rep}
	\Xi(M)=\E{\hat H^*M\hat H}\,,
\ee
where $\hat H:=(N^{-1/2}\hat x_{ij})$ (see \eqref{def: hat x}). Therefore, we only need to show that  for any unit vector $v$,
\ben
	\E{v^*\left(\eta+\hat H^*\tfrac{1}{2\ii}(M-M^*)\hat H\right)v}>0\,,
\ee
which is clearly true provided $\tfrac{1}{2\ii}(M-M^*)>0$.  
 
Next, $\Mp_N$ is closed under the map $M\mapsto-M^{-1}$  because
\ben
	\tfrac{1}{2\ii}(-M^{-1}+M^{-1*})=M^{-1}\left(\tfrac{1}{2\ii}(M-M^*)\right)M^{-1*}\,,
\ee
which is positive definite provided $\tfrac{1}{2\ii}(M-M^*)>0$.  In particular, if $M\in\M_N(\e,\de)$, then the above equation gives
\ben
	\tfrac{1}{2\ii}(-M^{-1}+M^{-1*})\geq \tfrac{\e}{\norma M^2}> \e\delta^{-2}\,.
\ee
Meanwhile, $\norma{M^{-1}}\leq \e^{-1}$ holds because 
\ben
	\inf_{\norma{v}=1}\absa{v^*Mv}\geq \inf_{\norma{v}=1}\absa{v^*\tfrac{1}{2\ii}(M-M^*)v}\geq \e\,.
\ee
The two estimates above yield $M\in\M_N({\frac{\e}{\de^2},\frac{1}{\e}})$.
\end{proof}
The following corollary tells us that the map $F$ not only maps $\Mp_N$ to itself, but also takes a compact subset of $\Mp_N$ to its 'strict' interior.
\begin{corollary}
Choose 
\be\label{def: e d}
	\delta=4\eta^{-1}\,,\e=\eta(K\de/2+\absa{z})^{-2}\,,\,.
\ee
The image of $\M_N(0,\de/2)$ under the map $F$ is contained in $\M_N(\e,\de/4)$, whose $\e$-neighborhood is a subset of $\M_N(0,\de/2)$.  
\end{corollary}
\begin{proof}
If $M\in\M_N(0,\de/2)$, then by \eqref{eq: xi rep},
\ben
	\norma{\Xi(M)} =\sup_{\norma{v}=1} \E {v^*\hat H^* M\hat H v}\leq \norma{M}\sup_{\norma{v}=1} \E {\norm{\hat Hv}^2}
\ee
Here $\E {\norm{\hat Hv}^2}=\tfrac{1}{N}\sum_{i,j,j'}\xi_{ijij'}v_j^*v_{j'} \leq K$.  Therefore 
\be\label{eq: Xi bd}
	\norma{\Xi(M)}\leq K\norma{M}\,.
\ee
It follows that $z+\Xi(M)\in \M_N(\eta, \absa{z}+K\de/2)$.  The last sentence of Lemma \ref{lem: closed} implies that $F(M)=-(z+\Xi(M))^{-1}\in \M_N(\eta(\absa{z}+K\de/2)^{-2}, \eta^{-1})$, which is exactly $\M_N(\e,\de/4)$ by our choice of $\e$ and $\de$. 
\end{proof}

The existence of a fixed point of $F$ follows from the following theorem:
 
 \begin{theorem}[Earle-Hamilton \cite{Earle70}]\label{thm: EH}
Let $\mathcal{D}$ be a nonempty domain in a complex
Banach space $X$ and let $h : \D \to\D$ be a bounded holomorphic function. If
$h(\D)$ lies strictly inside $\D$, in the sense that there is an $\e$ such that the $\e$-neighborhood of $h(\D)$ is a subset of $\D$, then $h$ has a unique fixed point in $\D$.
\end{theorem}

However, we need more than this theorem to get the stability of solutions.  Therefore, we prove the following lemma in our settings.  The proof is a slight modification of that of Theorem \ref{thm: EH} (see \cite{Harris03} for its proof).
\begin{lemma}\label{lem: contraction}
There is a metric $\dd$ on $\M_N(0,\de/2)$ such that the map $F$ defined in \eqref{def: F} is a strict contraction. In particular,
\ben
	\dd(F(Q_1),F(Q_2))\leq (1+\e/\de)^{-1}\dd(Q_1,Q_2),
\ee
where $\e,\de$ are defined in \eqref{def: e d}\,.  

The metric is equivalent to $\norma{\cdot}$ in the interior of $\M_N(0,\de/2)$ in the sense that it satisfies
\ben
	\dd(Q_1,Q_2)\geq \de^{-1} \norma{Q_1-Q_2}\,,\text{for $Q_1,Q_2\in\M_N(0,\de/2)$,}
\ee
and
\ben
	\dd(Q_1,Q_2)\leq \beta^{-1}\norma{Q_1-Q_2}\,,\text{for $Q_1,Q_2\in\M_N(\beta,\de/2-\beta)$\,,$\beta>0$.}
\ee

\end{lemma}

\begin{proof}
Let $\Delta$ be the unit disk in $\C$.  For any $Q\in\M_N(0,\de/2)$, $V\in \C^{N\times N}$, define
\ben
	\alpha(Q,V):=\sup\{\absa{Dg(Q)V}: g:\M_N(0,\de/2)\to\Delta\text{ holomorphic}\}\,,
\ee
where $D$ means differential. Here $\alpha(Q,V)$ is the so-called Caratheodory length of $V$ at point $Q$ in the domain $\M_N(0,\de/2)$.  It defines a norm on the tangent space at $Q$. Then one can define the length of a curve in $\M_N(0,\de/2)$ by integrating the length of tangent vectors.  For any piecewise smooth curve $\ga$ in $\M_N(0,\de/2)$, define the length of $\ga$ through
\ben
	L(\ga):=\int_0^1 \alpha(\ga(t),\ga'(t))\dd t\,.
\ee
Then we define the distance $\dd(Q_1,Q_2)$ between $Q_1$ and $Q_2$ by minimizing the length of curves connecting two points:
\ben
	\dd(Q_1,Q_2):=\inf_\ga \{L(\ga):\ga(0)=Q_1,\ga(1)=Q_2\}\,.
\ee
It is easy to check that $\dd$ is a metric. 
Now fix $Q_0\in\M_N(0,\de/2)$, define 
\ben
	\hat F(Q)=F(Q)+\e/\de(F(Q)-F(Q_0))\,,
\ee
which is a map from $\M_N(0,\de/2)$ to itself, because the diameter of $\M_N(0,\de/2)$ is at most $\de$.
Taking the diffrential at $Q_0$, $D\hat F(Q_0)=(1+\e/\de)DF(Q_0)$.
Let $g:\M_N(0,\de/2)\to\Delta$ be holomorphic. By the chain rule, for any $V\in\C^{N\times N}$, 
\ben
	D(g\circ\hat F)(Q_0)V=Dg(\hat F(Q_0))D\hat F(Q_0) V=(1+\e/\de)Dg(F(Q_0))DF(Q_0)V\,.
\ee
Since $g$ is arbitrary, by definition of $\alpha$,
\ben
	\alpha(F(Q_0),DF(Q_0)V)\leq (1+\e/\de)^{-1} \alpha(Q_0,V)\,.
\ee
This is true for any $Q_0\in\M_N(0,\de/2)$. Let $\ga$ be a smooth curve in $\M_N(0,\de/2)$, then for any $t\in[0,1]$ we set  $V=\ga'(t)$ and $Q_0=\ga(t)$ and see that
\ben
	\alpha(F(\ga(t)),DF(\ga(t))\ga'(t))\leq (1+\e/\de)^{-1}  \alpha(\ga(t),\ga'(t))\,.
\ee
Integrating over $t$, we have
\ben
	L(F(\ga))\leq (1+\e/\de)^{-1}  L(\ga)\,.
\ee
By definition of the metric $d$ we have
\ben
	\dd(F(Q_1),F(Q_2))\leq  (1+\e/\de)^{-1} \dd(Q_1,Q_2)\,.
\ee
Thus we have proved that $F$ is a strict contraction under $\dd$.  

For any $Q_0\in\M_N(0,\de/2)$, $V\in\M_N$, define a holomorphic map $g(Q):=\de^{-1} l (Q-Q_0)$, where $\l$ is in the dual space of $\M_N$ such that $\norma{l}=1$, $l(V)=\norma{V}$.  Then $g$ maps $\M_N(0,\de/2)$ into $\De$.  Since $Dg(Q_0)V=\de^{-1}l(V)$, we have $\alpha(Q_0,V)\geq \de^{-1} \norma{V}$.  It follows that for any piecewise smooth curve $\ga$, taking $Q_0=\ga(t)$ and $V=\ga'$ we have
\ben
	\alpha(\ga(t),\ga'(t))\geq \de^{-1}\norma{\ga'(t)}\,.
\ee
Integrating over $t$ yields
\ben
	L(\ga)\geq \de^{-1}\int_0^1\norma{\ga(1)-\ga(0)}\,.
\ee
Therefore, $\dd(Q_1,Q_2)\geq \de^{-1} \norma{Q_1-Q_2}$. 

On the other hand, for any $Q_0\in\M_N(\beta,\delta/2-\beta)$, $V\in\C^{N\times N}$ and holomorphic $g:\M_N(0,\de/2)\to \De$, define a holomorphic map $\phi: \De\to\De$ through
\ben
	\phi(\zeta):=g(Q_0+\beta\zeta V/\norma{V})\,.
\ee
By Schwartz-Pick theorem, $\absa{\phi'(0)}\leq 1$, therefore,
\ben
	\beta\norma{V}^{-1} \absa{Dg(Q_0)V}\leq 1\,.
\ee
Thus $\alpha(Q_0,V)\leq \beta^{-1}\norma{V}$.  It follows that for any piecewise smooth curve $\ga\subset \M_N(\beta,\delta/2-\beta)$, we have $\alpha(\ga,\ga')\leq \beta^{-1} \norma{\ga'}$.  By the convexity of $\M_N(\beta,\de/2-\beta)$, one can always find a curve $\ga\subset \M_N(\beta,\delta/2-\beta)$ that connects $Q_1$ and $Q_2$.  Integrating over $t$, 
\ben
	\dd(Q_1,Q_2)\leq \beta^{-1} \norma{Q_1-Q_2}\,.
\ee
\end{proof}

As consequences of Lemma \ref{lem: contraction}, we have the unique solvability of the equation \eqref{eq: sce n} as well as the stability of solutions.  The solution is obtained by iterating $F$ from an arbitrary initial point in $\Mp_N$.
\begin{theorem} \label{thm: solvable}
For each  $N\in\N$, $z\in\C^+$, equation \eqref{eq: sce n} has a unique solution in the class $\Mp_N$. The solution lies in the domain $\M_N(\eta(2K\eta^{-1}+\absa{z})^{-2},4\eta^{-1})$.
\end{theorem}
\begin{proof}
	Take any $Q_0\in \Mp_N$. Define recursively $Q_{k+1}:=F(Q_k)$.  It is easy to check that $Q_1\in \M_N(\e,\de/4)$.
	 Lemma \ref{lem: contraction} implies that $\dd(Q_{k+1},Q_k)\leq C_z (1+\e/\de)^{-k}$. Therefore $(Q_{k})_{k\geq 0}$ is a Cauchy sequence under the metric $d$.  Thus $Q_\infty:=\lim_k Q_k$ is the unique fixed point of $F$ on $\Mp_N$.
\end{proof}

The stability of solution follows from the fact that $F$ is a strict contraction under the metric $\dd$ and that $\dd$ is equivalent to $\norma{\cdot}$ in the interior of $\M_N(0,\de/2)$.
\begin{theorem}\label{thm: stable}
Let $\e$ and $\de$ be defined as in \eqref{def: e d}.  Suppose that there is an $M'\in\Mp_N$ such that 
\ben
	M'=F(M')+R\,,
\ee
where $R\in \M_N$ satisfies $\norma{R}\leq \e/2$. Let $A\in\Mp_N$ be the solution of \eqref{eq: sce n}, i.e., $M=F(M)$. Then
\be\label{eq: sce r}
	\norma{M'-M}\leq 2\de^2 \e^{-2} \norma{R}\,.
\ee
\end{theorem}
\begin{proof}
	Since $M'\in\Mp_N$, we have $z+\Xi(M')\in\M_N(\eta,\infty)$. The last sentence in Lemma \ref{lem: closed} yields $F(M')\in \M_N(0,\eta^{-1})$. It follows that $M'=F(M')+R\in \M_N(0,\de/2)$, which implies $F(M')\in\M_N(\e,\de/4)$, therefore $M'\in\M_N(\e/2,\de/2)$. Now that both $M'-R$ and $M$ are in the domain $\M_N(0,\de/2)$, we can take the distance between them
	\ben
		\dd (M,M'-R)=\dd(F(M),F(M'))\,.
	\ee
	By triangle inequality and the fact that $F$ is a strict contraction we have
	\ben
		\dd(M,M')-\dd(M',M'-R)\leq (1+\e/\de)^{-1}\dd(M,M')\,.
	\ee
	In other words,
	\ben
		\dd(M,M')\leq 2\de \e^{-1}\dd(M',M'-R)\,,
	\ee
	which implies, by the equivalence of $\dd$ and $\norma{\cdot}$,
	\ben
		\norma{M-M'}\leq 2\de^2\e^{-2}\norma{R}\,.
	\ee
\end{proof}

Theorem \ref{thm: stable} is not strong enough for our purpose, because we do not have an estimate on $\norma{R}$ , but only have an estimate for $\norma{R}_\infty:=\sup_{i,j}\absa{R_{ij}}$  as we did in Lemma \ref{lem: sce app}.   Thus we need a stronger stability theorem that only assumes the smallness of $\norma{R}_\infty$.  To do so, we need some other properties of the solution $M$.  Note that $\Xi(M)$ is a band matrix with band width $2K+1$ and $F(M)$ has a bounded condition number $\norma{F(M)}\norma{F(M)^{-1}}$,  we can prove the off-diagonal decay of $M$.   

\begin{theorem}\label{thm: decay}
Assume that $M$ solves equation \eqref{eq: sce n}. Let $\kappa(M)=\norma{M}\norma{M^{-1}}$ and $\alpha(M)= \left(\frac{\kappa(M)-1}{\kappa(M)+1}\right)^{\frac{2}{2K+1}}$.  Then,
\ben
	\absa{M_{ij}}\leq 2(2K+1)\kappa(M)\alpha(M)^{(\absa{i-j}-K)_+}\,.
\ee
\end{theorem}
\begin{proof}
This theorem is a immediate consequence of Lemma \ref{lem: demko} below,  which is a corollary of Theorem 2.4 in \cite{Demko84}.
\end{proof}

\begin{lemma}\label{lem: demko}
Let $A$ be an invertible finite or infinite matrix, which is $K$-banded in the sense that $A(i,j)=0$ given $\absa{i-j}>K$. Then,
\ben
	\absa{A^{-1}(i,j)}\leq 2(2K+1)\kappa(A) \alpha^{(\absa{i-j}-K)_+}\,.
\ee
Here $\kappa(A)=\norma{A}\norma{A^{-1}}$ and 
\ben
	\alpha=\left(\frac{\kappa(A)-1}{\kappa(A)+1}\right)^{\frac{2}{2K+1}}\,.
\ee
\end{lemma}
\begin{proof}
By the simple observation that $A^{-1}=A^*(AA^*)^{-1}$ we have
\ben
	A^{-1}(i,j)=\sum_{\absa{i-k}\leq K} A^*(i,k)(AA^*)^{-1}(k,j)\,.
\ee
Theorem \ref{thm: demko}, which we state below, implies that
\ben
	\absa{(A^*A)^{-1}(k,j)}\leq 2\norma{A^{-1}} \left(\frac{\sqrt{\kappa(AA^*)}-1}{\sqrt{\kappa(AA^*)}+1}\right)^{\frac{2\absa{k-j}}{2K+1}}\,.
\ee
Combing the above estimates with the trivial bound $\absa{A^*(i,k)}\leq \norma{A}$ and that $\kappa(AA^*)\leq \kappa(A)^2$, we have 
\ben
	\absa{A^{-1}(i,j)}\leq \sum_{\absa{k-i}\leq K}2\kappa(A) \alpha^{\absa{k-j}}\leq 2(2K+1) \kappa(A)\alpha^{(\absa{i-j}-K)_+}\,.
\ee
where $\alpha=\left(\frac{\kappa(A)-1}{\kappa(A)+1}\right)^{\frac{2}{2K+1}}$.
\end{proof}
We state below the theorem that we used in the proof.  The original theorem consists of two cases, where the matrix $A$ is positive definite or $A$ is not positive definite.  Our situation is the second case, but the result there is not as strong as we need, thus we do not directly apply it. We only state the first case of original theorem for our purpose.
\begin{theorem}[Demko-Moss-Smith \cite{Demko84}]\label{thm: demko}
Let $A$ and $A^{-1}$ be in $B(l^2(S))$ where $S=\{1,\dots,n\}$, $\Z_+$ or $\Z$. Then if $A$ is positive definite and has band width $m$ we have 
\ben
	\absa{A^{-1}(i,j)}\leq C\la^{\absa{i-j}}\,,
\ee
where 
\ben
	\la=\left(\frac{\sqrt{\kappa(A)-1}}{\sqrt{\kappa(A)+1}}\right)^{2/m}\,,
\ee
and 
\ben
	C=\norma{A^{-1}}\max\{1,(1+\sqrt{\kappa(A)})^2/(2\kappa(A))\}\,.
\ee
\end{theorem}

Besides the off-diagonal decay of $M$, another important observation is that the map $\Xi$ has a 'mollifying effect', as illustrated by Lemma \ref{lem: mollify} below. This lemma, although fairly simple, is very important because we will use it over and over again in the proof of stability.  In the sequel, $\norma{A}_\infty:=\sup_{i,j}\absa{A_{ij}}$ for any $A\in\C^{N\times N}$. 

\begin{lemma}\label{lem: mollify}
	For any $A\in\C^{N\times N}$, $\Xi(A)$ is a band matrix with band width $(2K+1)$ and satisfies $\norma{\Xi(A)}_\infty \leq (2K+1)\norma{A}$.
\end{lemma}
\begin{proof}
The claim that $\Xi(A)$ is a band matrix follows from the definition of the map $\Xi$.  By the assumption that $\var{x_{ij}}\leq 1$, we have $\absa{\xi_{ijkl}}\leq 1$, which implies
\ben
	\absa{(\Xi(A))_{ik}}\leq \tfrac{1}{N} \sum_{j,l}\absa{\xi_{ijkl}}\absa{A_{jl}}\leq \tfrac{1}{N}\sum_{\absa{j-l}\leq K} \norm{A}_\infty\,.
\ee
The conclusion of the lemma follows from the fact that $\norm{A}_\infty\leq\norm{A}$.
\end{proof}
Once Theorem \ref{thm: decay} and Lemma \ref{lem: mollify} are stablished, we are ready to prove a stronger stability theorem of equation \eqref{eq: sce n}, assuming only the smallness of $\norma{R}_\infty$.

\begin{theorem}\label{thm: stable 1}
Let $M$ solve equation \eqref{eq: sce n} and $M'\in \Mp_N$ solve the following perturbed equation
\ben
	M'=F(M')+R\,,
\ee
then there is an $\eps_z$ such that if $\norma{M'-M}_\infty\vee\norma{R}_\infty\leq \eps_z$, we have
\ben
	\norma{M'-M}_\infty \leq c_z\norma{R}_\infty\,
\ee
\end{theorem}

\begin{proof}

Let $\hat M'$ be defined as follows.  For $\absa{i-j}\leq K$ let $\hat M'_{ij}:=M'_{ij}$, otherwise let $\hat M'_{ij}:=F(M')_{ij}$.  Then $F(M')=F(\hat M')$, since the map $\Xi$ only depends on the near diagonal entries of $M'$. Therefore $\hat M'$ satisfies the equation
\ben
	\hat M'=F(\hat M')+ \hat R\,,
\ee
where $\hat R_{ij}=R_{ij}\ind{\absa{i-j}\leq K}$.  Now that $\hat R$ is a band matrix with band width $2K+1$, we have $\norm{\hat R}\leq (2K+1) \norma{R}_\infty$.  Therefore, when $\eps_z$ is small enough, Theorem \ref{thm: stable} implies that 
\be\label{eq: hat M}
	\norm{\hat M'-M}\leq c_z\norma{R}_\infty\,.
\ee
Therefore
\ben
	\norm{M'-M}_\infty\leq \norm{\hat M'-M'}_\infty+\norm{\hat M'-M}_\infty\leq c_z\norm{R}_\infty\,.
\ee
\end{proof}
%%%%%%%%%%%%%%%%%%%%%%%%  limit equation
%%%%%%%%%%%%%%%%%%%%%%%%
%%%%%%%%%%%%%%%%%%%%%%%%
\subsection{Solution to the limiting equation}\label{subsec: limit}
Recall that in Subsection \ref{subsec: sce} we defined $\cT=L^\infty([0,1),\cK)$, where $\cK$ is the space of bi-infinite sequences viewed as convolution operators on $l^2(\Z)$.  For the readers' convenience, we restate equation \eqref{eq: sce infty} below, recalling the definition \eqref{def: Psi} of $\Psi$.  The inverse in the space $\cT$ is defined in \eqref{def: inverse T}.
\ben
	m=(-z-\Psi(m))^{-1}\,.
\ee
We clarify some conventions here.  
\begin{definition}
For $f\in\cT$ and $\theta\in[0,1)$, we denote the norm of $f(\theta)$ in the space $\cK$ by $\norm{f(\theta)}$, while the notation $\norm{f}$ without the $\theta$ variable means the norm of $f$ in the space $\cT=L^\infty([0,1),\cK)$, i.e.,
\ben
	\norm{f}=\sup_{\theta\in[0,1)} \norm{f(\theta)}\,.
\ee
\end{definition}
%Now we summarize some properties of the norm $\norm{\cdot}$ on $\cT$:
%\begin{proposition}
%
%\end{proposition}
%
%\begin{proof}
%
%\end{proof}
%

Since we will also solve it by a fixed point argument, we define a map $\F$ on $\cT$ by
\ben
	\F(f):=(-z-\Psi(f))^{-1}\,.
\ee
The map $\F$ is not well-defined on the entire space, thus we introduce a subdomain of $\cT$ such that $\F$ is well-defined.  Recall that the inverse Fourier transform of $f\in\cT$ in the second variable is denoted by $\check f$. 
\begin{definition}
	 \[\cT_+:=\{f\in\cT: \inf_{\theta,\zeta\in[0,1) }\im \check f(\theta,\zeta)>0\}\]
	 and  \[\cT(\e,\de):=\{f\in\cT: \sup _{\theta,\zeta\in[0,1)} \absa{\check f(\theta,\zeta)}<\de, \inf_{\theta,\zeta\in[0,1)}\im \check f(\theta,\zeta)>\e\}\]
\end{definition}
It is still not obvious that $\F$ is well-defined on the domain $\cT_+$. Before showing that it is well-defined, we apply inverse Fourier transform to equation \eqref{eq: sce infty}. Denote $u(\theta,s):= \check m(\theta,s)$, then \eqref{eq: sce infty} becomes the following equation:
\be\label{eq: sce u}
	u(\theta,s)=(-z- Su(\theta,s))^{-1}\,,\forall (\theta,s)\in [0,1)^2\,.
\ee
where $S$ is an operator given by 
\be\label{def: S}
Su(\theta,s):=\iint \hat \psi(\theta,\phi,s,t) u(\phi,t)\dd \phi\dd t,
\ee
 and 
\ben
	\hat \psi(\theta,\phi,s,t):=\sum_{k,l} \psi(\theta,\phi,k,l) e^{\ii 2\pi(s k-t l)}\,.
\ee
From the definition of $S$, it is easy to see that for any $f\in\cT$, the inverse Fourier transform of $\F(f)$ satisfies
\be\label{eq: check F}
	\check{\F(f)}(\theta,s)=(-z-(S\check f)(\theta,s))^{-1}\,,\forall (\theta,s)\in[0,1)^2\,.
\ee
An important observation is that $\hat \psi(\theta,\phi,s,t)$ is non-negative; in particular, it has a positive lower bound if $\xi$ is positive definite in the sense of Definition \ref{def: pd}, as is justified by the following lemma:

\begin{lemma}\label{lem: positivity}
Assume that $\hat \psi$ is defined as above, then $0\leq \hat\psi(\theta,\phi,s,t) \leq K^2$. If $\xi$ is positive definite with lower bound $c_0$ in the sense of Definition \ref{def: pd}, then $\hat\psi (\theta,\phi,s,t)\geq c_0$ for all $(\theta,\phi,s,t)\in [0,1)^4$.
In particular, $\Psi$ is an bounded operator on $\cT$ satisfying $\norm{\Psi(f)}\leq K^2\norm{f}$ for any $f\in\cT$.
\end{lemma}
\begin{proof}
	Take an arbitrary real continuous function $g\in C([0,1]^2)$.  For each $N\in\N$ define a random variable $$Y_N= \tfrac{1}{N}\sum_{i,j} \hat x_{ij}g(i/N,j/N)e^{\ii 2\pi (s i-t j)}.$$
	Here $\hat x$ is defined as in \eqref{def: hat x}.  One can easily compute the variance of $Y_N$:
	\be
		\var{Y_N}=\frac{1}{N^2}\sum_{i,j,k,l} \psi(i/N,j/N,k,l)g(i/N,j/N) g((i+k)/N,(j+l)/N)e^{\ii 2\pi(sk-tl)}+\O{N^{-1}}\,.
	\ee
	Let $N\to \infty$ and use the fact that  $0\leq \var{Y_N}\leq K^2$, we have
	\be
		0\leq \iint g(\theta,\phi)^2 \hat\psi(\theta,\phi,s,t)\dd \theta\dd\phi\leq K^2\norma{g}_\infty^2.
	\ee
	Since $g$ is arbitrary, we conclude that $0\leq\hat \psi (\theta,\phi,s,t)\leq K^2$.  
	
	 If $\xi$ is positive definite, then $\var {Y_N}\geq c_0$.  Letting $N\to\infty$ we have $\hat\psi (\theta,\phi,s,t)\geq c_0$.
	 
	Finally, the claim that $\Psi$ is a bounded operator follows from the upper bound of $\hat\psi$.
\end{proof}
This lemma together with \eqref{eq: check F} immediately yields the following corollary:
\begin{corollary}\label{cor: Psi}
Assume that $\xi$ is positive definite with lower bound $c_0$ in the sense of Definition \ref{def: pd}. Suppose that $f\in\cT_+$ and $\im \tr f \geq \om$. Then $\F(f)\in \cT(c'_\om,c_\om)$.
\end{corollary}

We prove the unique solvability of equation \eqref{eq: sce infty}, using the same argument as in the proof of Theorem \ref{thm: solvable}.

\begin{theorem}\label{thm: solvable infty}
	For any $z\in\C^+$, there is a unique solution $m$ of \eqref{eq: sce infty} in the space $\cT_+$.  The solution  lies in $\cT(\e,\de/4)$ where
\ben	
	\de=4\eta^{-1}\,,\quad\e=\eta(K^2\de/2+\absa{z})^{-2}\,.
\ee
The solution $m$ is Liptchitz on each $I_j$ (see \eqref{def: partition}). 

The off-diagonal entries of $m$ decay exponentially:
	\ben
		\sup_{\theta}\absa{m(\theta,k)}\leq 2(2K+1) \kappa(m) \alpha ^{(k-K)_+}\,,
	\ee
	where $\kappa(m):=\norma{m}\norma{m^{-1}}$ and $\alpha:=\left((\kappa(m)-1)/(\kappa(m)+1)\right)^{2/(2K+1)}$.

The solution is stable in the sense that if $m'\in \cT(0,\de/2)$ satisfies a perturbed equation
\ben
	m'=\F(m')+r\,,
\ee
where $r\leq \e/2$, then $\norm{m'-m}\leq \de^2\e^{-2}\norm{r}$.
	
\end{theorem}
\begin{proof}
By Lemma \ref{lem: positivity} we see that if $\inf_{\theta,s}\check f(\theta,s)\geq 0$, then $(-z-S\check f(\theta,s))^{-1}\in (\e,\de/4)$ for all $(\theta,s)\in[0,1)^2$. Therefore, $\F$ maps $\cT(0,\de/2)$  to $\cT(\e,\de/4)$, whose $\e$-neighborhood is a subset of $\cT(0,\de/2)$.  Now one can repeat exactly the same argument in the proof of Lemma \ref{lem: contraction}, because the argument has nothing to do with the structure of $\cT$ except for that $\cT$ is a complex Banach space.  Therefore, there is a metric $\dd$ on $\cT(0,\de/2)$ under which the map $\F$ is a strict contraction satisfying \ben
	\dd(\F(f_1,f_2))\leq(1+\e/\de)^{-1}\dd(f_1,f_2)\,,
\ee
and
\ben 
	\de^{-1} \norma{f_1-f_2}\leq \dd(f_1,f_2)\leq \beta^{-1} \norma{f_1-f_2}\,,\text{ for $f_1,f_2\in\cT(\beta,\de/2-\beta)$}.
\ee
Thus one can obtain the solution by iterating $\F$ from an arbitrary initial point $Q_0\in\cT(0,\de/2)$.  The exponential decay of $m(\theta,k)$ follows from Lemma \ref{lem: demko}.  The stability follows from exactly the same argument in the proof of Theorem \ref{thm: stable}.
\end{proof}

%%%%%%%%%%%%%%%% bulk stability, continuum

Note that in the above theorem, the stability relies on the $\eta$ and $\absa{z}$.  If $\eta$ is too small or $\absa{z}$ is too large, the estimate breaks down.  The following theorem says that equation \eqref{eq: sce infty} is stable for $z\in\C^+$ as long as $\im \tr m(z)$ is positive, even if $z$ is very close to the real axis. The theorem follows from Theorem 2.12 of \cite{Ajanki2015a} and Lemma \ref{lem: positivity}.  For the sake of completeness, we give a self-contained proof in our setting. 
\begin{theorem}\label{thm: stable bulk infty}
Assume that $\xi$ is positive definite with lower bound $c_0>0$ in the sense of Definition \ref{def: pd}.  Let $m$ solve equation \eqref{eq: sce infty}.  Assume that $z$ and $\om$ satisfies $\im \tr m (z)\geq \om>0$, and that $m'$ satisfies the following perturbed equation:
\ben
	m'=(-z-\Psi(m'))^{-1}+r\,,
\ee
where $r\in\cT$.  Then there are $\e_\om$ and $c_\om$ such that 
\ben
	\norma{m'-m}\leq c_\om \norma{r}\,,
\ee
given $\norma{m'-m}\vee\norma{r}\leq \e_\om$.
\end{theorem}
\begin{proof}
Let $u$, $u'$ and $\check r$ be the inverse Fourier transform in the second variable of $m$, $m'$ and $r$, respectively.   Consider a space $\check\cT=L^\infty([0,1)^2)$ which is isometric to $\cT$ by Fourier transform.  Then, $u'\in\check\cT$ satisfies a perturbed version of \eqref{eq: sce u}:
\be\label{eq: sce u'}
	u'(\theta,s)=\frac{1}{-z-(Su')(\theta,s)}+\check r(\theta,s)\,,\forall (\theta, s)\in[0,1)^2\,.
\ee
In the sequel, we denote the norm on $\check\cT$ by $\norm{\cdot}_\infty$.  Then the assumptions of the theorem translated into $u'$ and $\check r$ reads
\ben
	\im \iint u(\theta,s)\dd\theta\dd s u\geq \om>0\,,
\ee
and
\ben	
	\norm{u'-u}_\infty \vee \norm{r}_\infty\leq \e_\om\,.
\ee 
In view of Corollary \ref{cor: Psi}, if $\e_\om$ is small enough, then $\inf_{\theta,s} \im u(\theta,s)\wedge \im u'(\theta,s)\geq c_\om$.

We take the imaginary part of equation \eqref{eq: sce u} and see that 
\ben
	\im u (\theta,s)= \frac{\eta+(S\im u)(\theta,s)}{\absa{z+(Su)(\theta,s)}^2}\,
\ee
Note that $\absa{z+(Su)(\theta,s)}^{-2}=\absa{u(\theta,s)}^2$, we have (here we omit the variables $(\theta,s)$)
\ben
	\im u=\absa{u}^2\left(\eta +(S\im u)\right)\,.
\ee
Dividing both sides by $\absa{u}$ and note that $\eta\geq0$ we have
\be\label{eq: imuu}
	\tfrac{\im u}{\absa{u}}\geq \absa{u} (S\im u)\,.
\ee
Now define an operator $T$ on $\cT$ by
\ben
	(Tf)(\theta,s):=\iint \absa{u(\theta,s)}\hat\psi(\theta,\phi,s,t)\absa{u(\phi,t)} f(\phi,t)\dd\phi\dd t\,.
\ee
Denote $w:=\tfrac{\im u}{\absa{u}}$. Recall the definition of $S$, then \eqref{eq: imuu} becomes
\ben	
	w=Tw\,.
\ee
Note that $T$ is an self-adjoint integral operator with a strictly positive integral kernel, by Krein-Rutman theorem, $T$'s largest eigenvalues is $1$ and has a positive spectral gap $\delta$ depending on $c_0$.

Now take the difference between \eqref{eq: sce u} and \eqref{eq: sce u'},
\be\label{eq: diff uu'}
	u-u'=\frac{S(u-u')}{(z+Su)^2}+\O{c_\om \norm{u-u'}_\infty^2}-\check r\,.
\ee
Note that $(z+Su)^{-2}=u^2$. Denote $q:=\frac{u-u'}{\absa{u}}$ and define a function $\alpha\in\check\cT$ through $e^{i\alpha}:=u/\absa{u}$, then
\be\label{eq: q}
	q=e^{i2\alpha}Tq+\O{c_\om \norm{u-u'}_\infty^2}-\check r\,.
\ee

We claim that $1-e^{i2\alpha}T$ is invertible and its inverse is bounded by some $C_\om$.  To prove this claim, it is sufficient to prove  $\re(v^* e^{\ii 2\alpha}Tv)\leq 1-c_\om$ for any unit vector $v\in L^2([0,1)^2)$ and some $c_\om>0$.  Write $v=v_1+v_2$ where $v_1$ is parallel to $w$ and $v_2$ is orthogonal to $w$. Then
\be\label{eq: re vv}
	\re(v^* e^{\ii 2\alpha}Tv)=\re(v^*e^{\ii 2\alpha}v_1)+\re(v^*e^{\ii 2\alpha}Tv_2)\leq \re(v^*_1e^{\ii 2\alpha}v_1) +2\norm{v_2}\,.
\ee
The first term on the right hand side by definition equals
\ben
	\re(v^*_1e^{\ii 2\alpha}v_1)=\iint \absa{v_1(\theta,s)}^2 \cos(2\alpha)\dd \theta\dd s\,.
\ee
Note that $\cos(2\alpha)\leq 1-c_\om$ since $\im u\geq \om$ and $\absa{u}\leq C_\om$, thus
\be
	\re(v^*_1e^{\ii 2\alpha}v_1)\leq (1-c_\om)\norm{v_1}^2\leq 1-c_\om\,.
\ee
Plugging into \eqref{eq: re vv} we have, 
\be\label{eq: re vv 1}
	\re(v^* e^{\ii 2\alpha}Tv)\leq 1-c_\om+2\norm{v_2}\,.
\ee
This estimate is useful when $\norm{v_2}$ is small.  On the other hand,
\be\label{eq: re vv 2}
	\re(v^*e^{\ii 2\alpha}Tv)\leq \norm{Tv}\leq \sqrt{\norm{v_1}^2+(1-\de)^2\norm{v_2}^2}\,.
\ee
Here we have used the spectral gap of $T$.  The right hand side is small when $\norm{v_2}$ is big.  Combining \eqref{eq: re vv 1} and \eqref{eq: re vv 2} we have, for some new constant $c_\om>0$,
\be
	\re(v^*e^{\ii 2\alpha}Tv)\leq 1-c_\om\,.
\ee
Thus we have proved the claim.  Therefore, \eqref{eq: q} yileds
\ben
	\norm{q}_{L^2}\leq c_\om(\norm{u-u'}_\infty^2+\norm{\check r}_\infty)\,.
\ee
Recall that $q=(u-u')/\absa{u}$, and that $\absa{u}$ is bounded, we have
\ben
	\norm{u-u'}_{L^2}\leq c_\om(\norm{u-u'}_{L^2}^2+\norm{\check r}_\infty)\,.
\ee
Take $\e_\om$ small enough, then we get
\ben
	\norm{u-u'}_{L^2}\leq c_\om\norm{\check r}_{\infty}\,.
\ee
Plugging in \eqref{eq: diff uu'} and recalling that $S$ is an integral operator with bounded integral kernel, we have
\ben
	\norm{u-u'}_{\infty}\leq c_\om\norm{\check r}_{\infty}\,.
\ee
By the definitions of $u$, $u'$ and $\check r$, the above estimate is equivalent to
\ben
	\norm{m-m'}\leq c_\om\norm{r}\,.
\ee
\end{proof}

%%%%%%%%%%%%%%%%%%%% converge golbal
%%%%%%%%%%%%%%%%%%%% converge golbal

One expects the solution $M$ of \eqref{eq: sce n} converges to the solution  $m$ of \eqref{eq: sce infty} as $N\to\infty$.  However, it is not clear how to define the limit of a sequence of band matrices whose sizes go to infinity.  Fortunately, we can show that $\tfrac{1}{N}\tr M$, as a sequence of holomorphic functions on $\C^+$, does converge to $\tr m$.  The trick is that one can 'imbed' $M$ into the space $\cT$ and show that this gives an approximate solution of \eqref{eq: sce infty} which is close to $m$.

\begin{theorem}\label{thm: converge}
Let $m$ solve \eqref{eq: sce infty} and $M$ solve \eqref{eq: sce n}.  Then, $\tr m-\tfrac{1}{N}\tr M=\O{c_z/N}$. Let $\hat M$ be the discretization of $m$ defined through $\hat M_{i,i+k}:= m(i/N,k)$, then
\ben
	\norm{M-F(\hat M)} \leq c_z/N\,.
\ee
\end{theorem}
\begin{proof}
	Define $\tm\in\cT$ by $\tm(\theta,k):=M_{\floor{N\theta},\floor{N\theta}+k}$ and $q\in\cT$ by $q(\theta,k):=(M^{-1})_{\floor{N\theta},\floor{N\theta}+k}$ if $\theta$'s $K/N$ neighborhood is contained in one of the $I_\alpha$'s (see \eqref{def: partition}) and $q(\theta,k):=z\de_{0k}$ otherwise. Recall that $M^{-1}=-z-\Xi(M)$, hence $M^{-1}$ has 
'continuity', i.e., $(M^{-1})_{i,k}-(M^{-1})_{i+1,k+1}=\O{N^{-1}}$ except for finitely many pairs $(i,k)$. 
	
	We claim that $-q$ is in the domain $\cT_+$. To prove this claim, we need to show that for any $(\theta,s)\in[0,1)^2$, $-\im \sum_k q(\theta,k)e^{\ii 2\pi sk}\geq 0$.  For $\theta$ whose $N^{-1/2}$- neighborhood is contained in one of the $I_\alpha$'s , take a vector
	\ben
		v=(v_k)= (e^{\ii 2\pi s  k}\ind{ \absa{Nk-\theta}\leq N^{-1/2}})\,.
	\ee
The components of $v$ vanishes outside the $N^{-1/2}$- neighborhood of $\floor{N\theta}$.  Now $\E{v^* \hat H^* M \hat Hv}= \tfrac{1}{N}\sum_{i',j,k,l}\xi_{i'jkl} M_{jl}  e^{\ii 2\pi s (i'-k)}\ind{\absa{Nk-\theta}\vee\absa{Ni'-\theta}\leq N^{-1/2}}$. Dividing by $\norm{v}^2$ and in view of the continuity of $M^{-1}$, we have
\ben
	\im \norm{v}^{-2}\E{v^* \hat H^* M \hat Hv}=-\eta-\im \sum_k q(\theta,k)e^{\ii 2\pi sk}+\O{c_z N^{-1/2}}\,.
\ee
The left hand side is non-negative because
\ben
	\im \E{v^* \hat H^* M \hat Hv}=\E{v^*\hat H^* \tfrac{1}{2\ii}(M+M^*)\hat Hv}\geq 0\,.
\ee
Therefore, $-\im\sum_k q(\theta,k)e^{\ii 2\pi sk}\geq \eta/2$ when $N$ is large enough.  For $\theta$ whose $N^{-1/2}$-neighborhood contains the endpoints of the $I_j$'s, the estimate also holds because of the continuity of $M^{-1}$.  Thus we have proved the claim.

Now we hope that $\tm$ is an approximate solution to \eqref{eq: sce infty}. However, it is not clear whether $\tm$ is in the domain $\cT_+$ or not.  Instead, $\F(\tm)\approx(-q)^{-1}$ is in the space $\cT_+$.  Thus, we prove that $\F(\tm)$ is an approximate solution, then we show that $\tr\F(\tm)$ is close to $m$.  

We first show that $\tm=\F(\tm)$ approximately holds. By definition
\be\label{eq: mq}
	\sum_l\tm(\theta,l)q(\theta,h-l)=\de_{0h}-w(\theta,h)\,,
\ee
where $w(\theta,h)=\sum_{\absa{l-h}\leq K}\tm(\theta,l)((M^{-1})_{\floor{N\theta+l},\floor{N\theta}+h}-q(\theta,h-l))$. Now we estimate $\norm{w(\theta)}$.  Note that $\norm{w(\theta)}\leq \sum_h\absa{w(\theta,h)}$, by the off-diagonal decay of $\tm$,
\ben
	\norm{w(\theta)}\leq c_z \sum_h \alpha_z^{(\absa{h}-2K)_+}\sup_{\absa{l-h}\leq K,j\in\N}\absa{((M^{-1})_{\floor{N\theta+l},\floor{N\theta}+l+j}-q(\theta,j)}\,.
\ee
Here $\alpha_z$ is a constant less than $1$ and depending on $z$.  Integrating over $\theta$, we have 
\be\label{eq: w bd}
	\int\norm{w(\theta)}\dd \theta \leq c_z/N\,.
\ee
 Now we estimate $q-(-z-\Psi(\tm))$, which is roughly the difference between an integral and its Riemann sum.  For $\theta$ whose $K/N$-neighborhood is contained in one of the $I_\alpha$'s,
\ben
	q(\theta,k)-(-z-\Psi(\tm))(\theta,k)=\iint \psi(\floor{N\theta}/N, \phi,k,l)\tm(\phi,l)\dd \phi\dd l - \tfrac{1}{N}\sum_{j,l}\xi_{\floor{N\theta},j,k,l}M_{jl}=\O{c_z/N}\,.
\ee
This combined with \eqref{eq: mq} and \eqref{eq: w bd} yields
\ben
	\int \norm{\tm(-z-\Psi(\tm))(\theta)}\dd \theta \leq c_z/N\,.
\ee
Remember that $-q\in\cT(\eta/2,c_z)$ when $N$ is large. Since $-q$ is close to $(z+\Psi(\tm))$, we have a bound $(z+\Psi(\tm))\in\cT(c'_z,c_z)$ when $N$ is large.  Therefore, it is comfortable to take the inverse of $(-z-\Psi(\tm))$, which is $\F(\tm)$.  Moreover, $\F(\tm)$ has exponential decaying off-diagonal entries.  Therefore, we multiply $\F(\tm)$ to the estimate above and see
\be\label{eq: m Fm 0}
	\int\norm{\tm(\theta)-\F(\tm)(\theta)}\dd\theta\leq c_z/N\,.
\ee
Then we apply the map $f\mapsto -z-\Psi(f)$ to $\tm$ and $\F(\tm)$, recalling that $\Psi$ is an integral operator with bounded integral kernel,
\ben
	\norm{(-z-\Psi(\tm))-(-z-\Psi(\F(\tm)))}\leq c_z/N\,.
\ee
Now we take the inverse of $(-z-\Psi(\tm))$ and $(-z-\Psi(\F(\tm)))$.  Because everything is bounded, we have 
\ben
	\norm{\F(\tm)-\F(\F(\tm))}\leq c_z/N\,.
\ee
To this end, we see that $\F(\tm)$ satisfies a perturbed equation $\F(\tm)=\F(\F(\tm))+$ error terms.  Apply the stability part in Theorem \ref{thm: solvable infty} to conclude
\be\label{eq: Ftm 0}
	\norm{\F(\tm)-m}\leq c_z/N\,.
\ee
Combined with \eqref{eq: m Fm 0} and note that $\tr \tm =\tfrac{1}{N} \tr M$ yields
\ben
	\tr m-\tfrac{1}{N} \tr M=\O{c_z/N}\,.
\ee
Thus we have proved the first claim in the theorem.

	To prove the last claim of the theorem, we combine \eqref{eq: m Fm 0} and \eqref{eq: Ftm 0} and see
\ben
	\int\norm{\tm(\theta)-m(\theta)}\dd \theta\leq c_z/N\,.
\ee
Let $\hat M$ be the discretization of $m$ defined through $\hat M_{i,i+k}:= m(i/N,k)$, we see that
\ben
	\sum_{i}\absa{M_{i,i+k}-\hat M_{i,i+k}}\leq c_z/N\,,\forall k.
\ee 
Applying $\Xi$ to $M$ and $\hat M$, we get
\ben
	\norm{\Xi(M)-\Xi(\hat M)}\leq c_z/N\,.
\ee
This enables us to estimate
\ben
	\norm{F(M)-F(\hat M)} \leq \norm{F(M)(\Xi(M)-\Xi(\hat M))F(\hat M)}\,.
\ee
Since $F(M)$ and $F(\hat M)$ are bounded, we have
\ben
	\norm{F(M)-F(\hat M)}\leq c_z \norm{\Xi(M)-\Xi(\hat M)}\leq c_z/N\,.
\ee
Thus we have proved the last claim in the theorem
\ben
	\norm{M-F(\hat M)} \leq c_z/N\,.
\ee
\end{proof}

%%%%%%%%%%%%%%%% bulk stability
%%%%%%%%%%%%%%%% bulk stability
%%%%%%%%%%%%%%%% bulk stability

\subsection{Stability in the bulk}\label{subsec: stable bulk}
In Theorem \ref{thm: stable bulk infty} we see that if  $\im\tr m(z)$ is bounded below, then the solution $m$ of \eqref{eq: sce infty} is stable under small perturbations, even if $z$ is close to the real axis. In this subsection we show that under the same assumption that $\im\tr m(z)$ is bounded below, the solution $M$ to the finite-demensional equation \eqref{eq: sce n} is also stable (Theorem \ref{thm: stable bulk}).  The strategy is to show that one can approximate $m$ by 'imbedding' $M$ into the space $\cT$.  The stability of $m$ will imply the stability of $M$.  Before proving the stability, we also prove Theorem \ref{thm: converge local} which says that $\tfrac{1}{N}\tr M$ converges to $\tr m$ in any domain where $\im \tr m$ is bounded below.  Theorem \ref{thm: converge local} will be used in the proof of Theorem \ref{thm: stable bulk}.

We will need the following property of the map $\Xi$.
\begin{lemma}\label{lem: smooth}
Assume that $\xi$ is positive definite with lower bound $c_0>0$ in the sense of Definition \ref{def: pd}.   Then, for any Hermitian matrix $A$ with $A\geq 0$ and $\tfrac{1}{N}\tr A\geq 1$, one has
\ben
	 \Xi(A)\geq c_0\,.
\ee
\end{lemma}
\begin{proof}
Assume that $A$ has spectral decomposition $A_{ij}=\sum_\al u_i^\alpha \ol{u_j^\alpha} \la_\al$. Then for any $\norma{v}=1$,
\ben
	v^*\Xi(A)v=\tfrac{1}{N}\sum_\al\left(\sum_{i,j,k,l} \ol{v_i} v_k\xi_{ijkl} u_k^\alpha \ol{u_l^\alpha} \la_\al\right) \,.
\ee
Note that $\sum_{i,j,k,l} \ol{v_i} v_k\xi_{ijkl} u_j^\alpha \ol{u_l^\alpha} $ is the variance of $\sum_{i,j} \hat x_{ij} \ol{v_i}u_j^\alpha$ where $\hat{x}_{ij}$ is defined in \eqref{def: hat x}.  By the assumption that $\xi$ is positive definite, $\sum_{i,j,k,l} \ol{v_i} v_k\xi_{ijkl} u_j^\alpha \ol{u_l^\alpha} $ is bounded below by $c_0$. It follows that 
\ben
	v^*\Xi(A)v\geq c_0\,.
\ee
\end{proof}

\begin{corollary}\label{cor: Xi}
Assume that $\xi$ is positive definite with lower bound $c_0>0$ in the sense of Definition \ref{def: pd}.  Suppose $\absa{z}\leq \om^{-1}$  and $Q\in\Mp_N$ satisfies $\im \tr Q\geq \om >0$, then there is a $c_\om$ such that $F(Q)\in \M_n(c_\om, c_\om^{-1})$. 

In particular, assume $M$ is the solution to equation \eqref{eq: sce n}.  Suppose $\absa{z}\leq \om^{-1}$ and $\im \tr M\geq \om>0$.  Then $M\in \cT(c_\om, c_\om^{-1})$, $\norm{M^{-1}}\leq C_\om$.
\end{corollary}
\begin{proof}
By Lemma \ref{lem: smooth}, $\im \tr Q\geq \om$ and $Q\in\Mp_N$ implies
\ben
	\tfrac{1}{2i}\Xi(Q-Q^*)\geq c_0\om\,,
\ee
i.e., $\Xi(Q)\in\M_N(c_0\om,+\infty)$. Therefore $F(Q)\in\M_N(0,(c_0\om)^{-1})$. This implies that $\Xi(Q)\in\M_N(c_0\om,K(c_0\om)^{-1})$ by the boundedness of the map $\Xi$.  Thus $F(Q)\in\M_n(c_\om, c_\om^{-1})$ for some constant $c_\om$.

\end{proof}

%%%%%%%%%%%%%%%%%%%% converge local
%%%%%%%%%%%%%%%%%%%% converge local
The following theorem says when $\im \tr m(z)$ is bounded below, then the solution $M$ to \eqref{eq: sce n} converges to $m$ in a certain sense.  Again the strategy is similar to Theorem \ref{thm: converge}, that is, we 'imbed' $M$ into the space $\cT$ and show that this gives an approximate solution that is close to $m$.  However, one needs to use a bootstrapping argument, because we will use Theorem \ref{thm: stable bulk infty} that assume the smallness of $\norm{m'-m}$ where $m'$ is the approximate solution constructed from $M$.  We do not have an a priori bound for $\norm{m'-m}$ when $z$ is close to the real axis. Therefore we need to start with $z$ far from the real axis, then iteratively get the bound close to the real axis.
\begin{theorem}\label{thm: converge local}
Assume that $\xi$ is positive definite in the sense of Definition \ref{def: pd}.   Let $m$ solve equation \eqref{eq: sce infty} and $M$ solve equation \eqref{eq: sce n}.  Fix a domain $\D\subset \C^+$ such that $\im \tr m(z)$ is bounded below by $\om>0$ on $\D$.  Then, uniformly in $\D$,
\ben
	\tr m-\tfrac{1}{N} \tr M=\O{c_\D/N}\,.
\ee
Let $\hat M$ be the discretization of $m$ defined through $\hat M_{i,i+k}:= m(i/N,k)$, then
\ben
	\norm{M-F(\hat M)} \leq c_\D/N\,.
\ee
	
\end{theorem}

\begin{proof}
Fix a $z\in\D$ such that $\absa{\tr m-\tfrac{1}{N}\tr M}\leq N^{-1/2}$.  Therefore, $\tfrac{1}{N}\tr M\geq \om/2$ when $N$ is large. Such a $z$ exists by Theorem \ref{thm: converge}.

Define $\tm\in\cT$ by $\tm(\theta,k)=M_{\floor{N\theta},\floor{N\theta}+k}$ and $q\in\cT$ by $q(\theta,k)=(M^{-1})_{\floor{N\theta},\floor{N\theta}+k}$  if $\theta$'s $K/N$ neighborhood is contained in one of the $I_\alpha$'s (see \eqref{def: partition}) and $q(\theta,k):=i\om\de_{0k}$ otherwise. 
	We claim that $-q$ is in the domain $\cT(c_\om,c_\om^{-1})$ for some $c_\om>0$. To prove this claim, it is sufficient to show that for any $\theta,s\in[0,1)$, $-\im \sum_k q(\theta,k)e^{\ii 2\pi sk}\geq c_\om$.  For $\theta$ whose $N^{-1/2}$- neighborhood is contained in one of the $I_\alpha$'s, take a vector
	\ben
		v=(v_k)= (e^{\ii 2\pi s  k}\ind{ \absa{Nk-\theta}\leq N^{-1/2}})\,.
	\ee
The components of $v$ vanishes outside the $N^{-1/2}$- neighborhood of $\floor{N\theta}$.  Now $\E{v^* \hat H^* M \hat Hv}= \tfrac{1}{N}\sum_{i',j,k,l}\xi_{i'jkl} M_{jl}  e^{\ii 2\pi s (i'-k)}\ind{\absa{Nk-\theta}\vee\absa{Ni'-\theta}\leq N^{-1/2}}$. Dividing by $\norm{v}^2$ and in view of the continuity of $M^{-1}$, we have
\ben
	\norm{v}^{-2}\E{v^* \hat H^* M \hat Hv}=-z-\sum_k q(\theta,k)e^{\ii 2\pi sk}+\O{c_\D N^{-1/2}}\,.
\ee
The left hand side has a positive imaginary part, indeed,
\ben
	 \norm{v}^{-2}\im\E{v^* \hat H^* M \hat Hv}=\norm{v}^{-2}\im\E{v^* \hat H^* \tfrac{1}{2\ii}(M-M^*) Hv}\geq c_\om\,.
\ee
Therefore, $-\im\sum_k q(\theta,k)e^{\ii 2\pi sk}\geq c_\om$ when $N$ is big enough.  For $\theta$ whose $N^{-1/2}$-neighborhood contains the endpoints of the $I_j$'s, the estimate also holds because of the continuity of $M^{-1}$.  Thus we have proved the claim.

Now we hope that $\tm$ is an approximate solution to \eqref{eq: sce infty}. However, it is not clear whether $\tm$ is in the domain $\cT_+$ or not.  Instead we prove that $\F(\tm)$ is an approximate solution, then show that $\tr\F(\tm)$ is close to $m$.  

We first show that $\tm=\F(\tm)$ approximately holds. By definition,
\be\label{eq: mq 0}
	\sum_l\tm(\theta,l)q(\theta,h-l)=\de_{0h}-w(\theta,h)\,,
\ee
where $w(\theta,h)=\sum_{\absa{l-h}\leq K}\tm(\theta,l)((M^{-1})_{\floor{N\theta+l},\floor{N\theta}+h}-q(\theta,h-l))$. Now we estimate $\norm{w(\theta)}$.  Note that $\norm{w(\theta)}\leq \sum_h\absa{w(\theta,h)}$, by the off-diagonal decay of $\tm$,
\ben
	\norm{w(\theta)}\leq c_\D \sum_h \alpha_\D^{(\absa{h}-2K)_+}\sup_{\absa{l-h}\leq K,j\in\N}\absa{((M^{-1})_{\floor{N\theta+l},\floor{N\theta}+l+j}-q(\theta,j)}\,.
\ee
Integrating over $\theta$, we have 
\be\label{eq: w bd 0}
	\int\norm{w(\theta)}\dd \theta \leq c_\D/N\,.
\ee
Now we estimate $q-(-z-\Psi(\tm))$, which is roughly the difference between an integral and its Riemann sum. For $\theta$ whose $K/N$-neighborhood is contained in one of the $I_\alpha$'s,
\be\label{eq: q psitm}
	q(\theta,k)-(-z-\Psi(\tm))(\theta,k)=\iint \psi(\floor{N\theta}/N, \phi,k,l)\tm(\phi,l)\dd \phi\dd l - \tfrac{1}{N}\sum_{j,l}\xi_{\floor{N\theta},j,k,l}M_{jl}=\O{c_\D/N}\,.
\ee
This combined with \eqref{eq: mq 0} and \eqref{eq: w bd 0} yields
\ben
	\int \norm{\tm(-z-\Psi(\tm))(\theta)}\dd \theta \leq c_\D/N\,.
\ee
Recall that $-q\in\cT(c'_\D,c_\D)$ when $N$ is large. Since $-q$ is close to $(z+\Psi(\tm))$ we have a bound $(z+\Psi(\tm))\in\cT(c'_\D,c_\D)$ when $N$ is large.  Therefore, it is comfortable to take the inverse of $-(z+\Psi(\tm))$, which is $\F(\tm)$.  Moreover, $\F(\tm)$ has exponential decaying off-diagonal entries.  Therefore, we multiply $\F(\tm)$ to the estimate above and see
\be\label{eq: m Fm}
	\int\norm{\tm-\F(\tm)}\dd\theta\leq c_\D/N\,.
\ee
Applying the map $f\mapsto -z-\Psi(f)$ to both $\tm$ and $\F(\tm)$, 
\be\label{eq: psitm}
	\norm{(-z-\Psi(\tm))-(-z-\Psi(\F(\tm)))}\leq c_\D/N\,.
\ee
Now we take the inverse of $(-z-\Psi(\tm))$ and $(-z-\Psi(\F(\tm)))$.  Because everything is bounded, we have 
\ben
	\norm{\F(\tm)-\F(\F(\tm))}\leq c_\D/N\,.
\ee
Apply Theorem \ref{thm: stable bulk infty} to conclude
\be\label{eq: Ftm}
	\norm{\F(\tm)-m}\leq c_\D/N\,.
\ee
So far we have obtained two estimates \eqref{eq: m Fm} and \eqref{eq: Ftm} assuming that  $\absa{\tr m-\tfrac{1}{N}\tr M}\leq N^{-1/2}$.

Now define $\D_N:=\{z\in\D:\norm{ \F(\tm)-m}\vee\absa{\tr m-\tfrac{1}{N}\tr M}\leq N^{-1/2}\}$.  On this domain, the above argument holds and we get \eqref{eq: Ftm}.  
Combined with \eqref{eq: m Fm} and note that $\tr \tm =\tfrac{1}{N} \tr M$ we get
\ben
	\tr m-\tfrac{1}{N} \tr M=\O{c_\D/N}\,.
\ee
Therefore, we see that when $N$ is large enough, the quantity $\norm{ \F(\tm)-m}\vee\absa{\tr m-\tfrac{1}{N}\tr M}$ cannot be in the interval $(c_\D/N, N^{-1/3})$.  Note that the quantity is continuous in $z$, therefore, it is either above $N^{-1/3}$ for all $z\in\D$ or below $c_\D/N$ for all $z\in\D$.  The latter case is true, because for any fixed $z\in\D$, $\norm{ \F(\tm)-m}\vee\absa{\tr m-\tfrac{1}{N}\tr M}\leq c_z/N$.  Thus we have proved the first claim in the theorem.

To prove the last claim of the theorem, combining \eqref{eq: m Fm} and \eqref{eq: Ftm}, we have
\ben
	\int\norm{\tm(\theta)-m(\theta)}\dd \theta\leq c_\D/N\,.
\ee
Let $\hat M$ be the discretization of $m$ defined through $\hat M_{i,i+k}:= m(i/N,k)$, we see that
\ben
	\sum_{i}\absa{M_{i,i+k}-\hat M_{i,i+k}}\leq c_\D/N\,,\forall k.
\ee 
Applying $\Xi$ to $M$ and $\hat M$, we get
\ben
	\norm{\Xi(M)-\Xi(\hat M)}\leq c_\D/N\,.
\ee
This enables us to estimate
\ben
	\norm{F(M)-F(\hat M)} \leq \norm{F(M)(\Xi(M)-\Xi(\hat M))F(\hat M)}\,.
\ee
Since $F(M)$ and $F(\hat M)$ are bounded by Corollary \ref{cor: Xi}, we have
\ben
	\norm{F(M)-F(\hat M)}\leq c_\D \norm{\Xi(M)-\Xi(\hat M)}\leq c_\D/N\,.
\ee
Thus we have proved the last claim in the theorem
\ben
	\norm{M-F(\hat M)} \leq c_\D/N\,.
\ee
\end{proof}

%%%%%%%%%%%%%%%%%%%% bulk stable
%%%%%%%%%%%%%%%%%%%% bulk stable

Now we are ready to prove the stability in the bulk. As mentioned in the beginning of this subsection, our strategy is very similar to that of Theorem \ref{thm: converge local}, that is, we 'imbed' $M$, $M'$ and $R$ into the space $\cT$, then apply Theorem \ref{thm: stable bulk infty} to get a bound on $\norma{M-M'}_\infty$.  
	
\begin{theorem}\label{thm: stable bulk}
Assume that $\xi$ is positive definite in the sense of Definition \ref{def: pd}.   Let $m$ solve equation \eqref{eq: sce infty} and $M$ solve equation \eqref{eq: sce n}.  Fix a domain $\D\subset \C^+$ such that $\im \tr m(z)$ is bounded below by $\om>0$ on $\D$.   Assume that $M'$ solve the following perturbed equation
\ben
	M'=F(M')+R\,,
\ee
then there are $\eps_\D$ and $N_\D$ such that if $\norma{M'-M}_\infty \vee\norma{R}_\infty\leq \eps_\D$, we have
\ben
	\norma{M'-M}_\infty \leq c_\D( \norma{R}_\infty+N^{-1})\,
\ee
for $N\geq N_\D$. 

\end{theorem}

\begin{proof}
At the end of this proof, we need the quantity $\tm$ that was defined in the proof of Theorem \ref{thm: converge local} and some estimates obtained in that proof.  First we prove the theorem under the additional assumption that $R$ is an band matrix with band width $(2K+1)$.  Later on we will remove this additional assumption.

We first get bounds on $F(M')$.  When $\e_\D$ is small enough, we have $\tfrac{1}{N}\tr M' =\tfrac{1}{N}\tr M+\O{\e_\D}$, which implies $\tfrac{1}{N}\tr M'\geq \om-\e_\D-c_\D N^{-1}$ by Theorem \ref{thm: converge local}. Therefore, $\tfrac{1}{N}\tr M'\geq \om/2$ if we take $\e_\D$ small enough and $N$ large enough.  It follows from Corollary \ref{cor: Xi} that $F(M')\in \M_N(c'_\D,c_\D)$.    

Define $m'\in\cT$ by  $m'(\theta,k)=M'_{\floor{N\theta},\floor{N\theta}+k}$ and $q\in\cT$ by $q(\theta,k)=(F(M')^{-1})_{\floor{N\theta},\floor{N\theta}+k}$  if $\theta$'s $K/N$ neighborhood is contained in one of the $I_\alpha$'s (see \eqref{def: partition}) and $q(\theta,k):=i\om\de_{0k}$ otherwise. 
	We claim that $-q$ is in the domain $\cT(c_\D,C_\D)$. To prove this claim, we need to show that $-q\in \cT(c_\D,+\infty)$, that is, for any $\theta,s\in[0,1)$, $-\im \sum_k q(\theta,k)e^{\ii 2\pi sk}\geq c_\D$.  For $\theta$ whose $N^{-1/2}$- neighborhood is contained in one of the $I_\alpha$'s, take a vector
	\ben
		v=(v_k)= (e^{\ii 2\pi s  k}\ind{ \absa{Nk-\theta}\leq N^{-1/2}})\,.
	\ee
The components of $v$ vanishes outside the $N^{-1/2}$- neighborhood of $\floor{N\theta}$.  Now $\E{v^* \hat H^* M'Hv}= \sum_{i',j,k,l}\xi_{i'jkl} M'_{jl} e^{\ii 2\pi s (i'-k)}\ind{\absa{Nk-\theta}\vee\absa{Ni'-\theta}\leq N^{-1/2}}$. Dividing by $\norm{v}^2$ and in view of the continuity of $\Xi(M')$, we have
\ben
	\norm{v}^{-2}\E{v^* \hat H^* M \hat Hv}=-z-\sum_k q(\theta,k)e^{\ii 2\pi sk}+\O{c_\D N^{-1/2}}\,.
\ee
The left hand side has a positive imaginary part because
\ben
	\norm{v}^{-2}\im \E{v^* \hat H^* M \hat Hv}=\norm{v}^{-2}\E{v^* \hat H^* \tfrac{1}{2}(M-M^*)\hat Hv}\geq c_0\,.
\ee
Therefore, $-\im \sum_k q(\theta,k)e^{\ii 2\pi sk}\geq c_\D$ when $N$ is big enough.  For $\theta$ whose $N^{-1/2}$-neighborhood contains the endpoints of the $I_j$'s, the estimate also holds because of the continuity of $\Xi(M')$.  Thus we have proved the claim.

Next, we prove that $\F(m')$ is an approximate solution.  We first show that $m'=\F(m')$ approximately holds. By definition
\be\label{eq: mq 1}
	\sum_lm'(\theta,l)q(\theta,h-l)=\de_{0h}-w(\theta,h)\,,
\ee
where $w(\theta,h)=\sum_{\absa{l-h}\leq K}m'(\theta,l)((M^{-1})_{\floor{N\theta+l},\floor{N\theta}+h}-q(\theta,h-l))$. Now we estimate $\norm{w(\theta)}$.  Note that $\norm{w(\theta)}\leq \sum_h\absa{w(\theta,h)}$, by the off-diagonal decay of $m'$,
\ben
	\norm{w(\theta)}\leq c_\D\sum_h \alpha_\D^{(\absa{h}-2K)_+}\sup_{\absa{l-h}\leq K,j\in\N}\absa{((M^{-1})_{\floor{N\theta+l},\floor{N\theta}+l+j}-q(\theta,j)}\,.
\ee
Integrating over $\theta$, we have 
\be\label{eq: w bd 1}
	\int\norm{w(\theta)}\dd \theta \leq c_\D(N^{-1}+\norm{R}_\infty)\,.
\ee
Now we estimate $q-(-z-\Psi(m'))$, which is roughly the difference between an integral and its Riemann sum. For $\theta$ whose $K/N$-neighborhood is contained in one of the $I_\alpha$'s,
\ben
	q(\theta,k)-(-z-\Psi(m'))(\theta,k)=\iint \psi(\floor{N\theta}/N, \phi,k,l)m'(\phi,l)\dd \phi\dd l - \tfrac{1}{N}\sum_{j,l}\xi_{\floor{N\theta},j,k,l}M_{jl}=\O{c_\D/N}\,.
\ee
This combined with \eqref{eq: mq 1} and \eqref{eq: w bd 1} yields
\ben
	\int \norm{m'(-z-\Psi(m'))(\theta)}\dd \theta \leq c_\D(N^{-1}+\norm{R}_\infty)\,.
\ee
Recall that $q\in\cT(c'_\D,c_\D)$ when $N$ is large. Since $-q$ is close to $(z+\Psi(m'))$ we have a bound $(z+\Psi(m'))\in\cT(c'_\D,c_\D)$ when $N$ is large.  Therefore, it is comfortable to take the inverse of $-(z+\Psi(m'))$, which is $\F(m')$.  Moreover, $\F(m')$ has exponential decaying off-diagonal entries.  Therefore, we multiply $\F(m')$ to the estimate above and see
\be\label{eq: m Fm'}
	\int\norm{m'(\theta)-\F(m')(\theta)}\dd\theta\leq c_\D(N^{-1}+\norm{R}_\infty)\,.
\ee
Applying the map $f\mapsto -z-\Psi(f)$ to both $m'$ and $\F(m')$,
\ben
	\norm{(-z-\Psi(m'))-(-z-\Psi(\F(m')))}\leq c_\D(N^{-1}+\norm{R}_\infty)\,.
\ee
Now we take the inverse of $(-z-\Psi(m'))$ and $(-z-\Psi(\F(m')))$.  Because everything is bounded, we have 
\be\label{eq: FFm'}
	\norm{\F(m')-\F(\F(m'))}\leq c_\D(N^{-1}+\norm{R}_\infty)\,.
\ee
In order to apply Theorem \ref{thm: stable bulk infty}, we need a bound for $\norm{\F(m')-m}$.  At the end of the proof of Theorem \ref{thm: converge local}, we showed \eqref{eq: Ftm} that $\norm{\F(\tm)-m}\leq c_\D/N$.  Meanwhile, it is easy to get a bound for $\norm{\F(m')-\F(\tm)}$:
\ben
	\norm{\F(m')-\F(\tm)}\leq c_\D \norm{M'-M}_\infty\leq c_\D\e_\D\,.
\ee
Therefore 
\be
	\norm{\F(m')-m}\leq \norm{\F(\tm)-m}+\norm{\F(m')-\F(\tm)}\leq c_\D (\e_\D+N^{-1})\,.
\ee
Take $\e_\D$ small enough, then apply Theorem \ref{thm: stable bulk infty} to \eqref{eq: FFm'} to get 
\ben
	\norm{\F(m')-m}\leq c_\D(N^{-1}+\norm{R}_\infty)\,.
\ee
Thus $\norm{\F(m')-\F(\tm)}\leq c_\D(N^{-1}+\norm{R}_\infty)$. Recall \eqref{eq: m Fm} and \eqref{eq: m Fm'}, we have
\ben
	\int \norm{m'(\theta)-\tm(\theta)}\dd \theta \leq c_\D(N^{-1}+\norm{R}_\infty)\,.
\ee
By definition of $m'$ and $\tm$, 
\ben
	\norm{\Xi(M')-\Xi(M)}_\infty c_\D(N^{-1}+\norm{R}_\infty)\,.
\ee
Note that $\norm{M'-M}_\infty=\norm{F(M')(\Xi(M)-\Xi(M')M+R}_\infty $, thus
\ben
	\norm{M'-M}_\infty\leq c_\D\norm{\Xi(M')-\Xi(M)} +\norma{R}_\infty \leq c_\D(N^{-1}+\norm{R}_\infty)\,.
\ee
Here we have used the exponential decay of the off-diagonals of $M'$ and $M$.  So far we have proved the theorem under the additional assumption that $R$ is a $(2K+1)$-banded matrix.
	
Now we remove the assumption that $R$ is a $(2K+1)$-banded matrix. Define $M^*$ by $M^*_{ij}:=M'_{ij}$ for $\absa{i-j}\leq K$ and $M^*_{ij}:=F(M')_{ij}$ otherwise.  Then $F(M^*)=F(M')$ because the map $F$ only depends on the near diagonal entries.  Therefore 
\ben
	M^*=F(M^*)+R^*\,,
\ee
where $R^*_{ij}=R_{ij}\ind{\absa{i-j}\leq K}$.  This reduces to the case where we assume that the error is a band matrix, hence $\norm{M^*-M}_\infty\leq c_\D(N^{-1}+\norm{R}_\infty)$. Thus,
\ben
	\norm{M'-M}_\infty\leq \norm{M^*-M}_\infty+\norm{M'-M^*}_\infty \leq c_\D(N^{-1}+\norm{R}_\infty)\,.
\ee
\end{proof}

%%%%%%%%%%%%%%%%%%%%%%%%  global law
%%%%%%%%%%%%%%%%%%%%%%%%
%%%%%%%%%%%%%%%%%%%%%%%%
\section{Proof of the global law and local law}\label{sec: law}
\subsection{Proof of the global law}
In order to apply Lemma \ref{lem: sce app}, we need estimates for $\Ga$ and $\ga$ for fixed $z$:
\begin{lemma}\label{lem: gG}
Recall that $q=N^{\tau}$. For any $0<\sigma\leq \tau/2$, $p\geq 2$, we have
\ben
	\ga\vee\Ga=\O{ \eta^{-1}+\absa{z}}\,,
\ee
with probability $1-c_pN^{-\sigma p+s}$.  Here $s$ is a universal constant.
\end{lemma}
\begin{proof}
Clearly $\Ga\leq \eta^{-1}$. In order to bound $\ga$, we denote $\K:=\I\cup\J$ for any $\I,\J\subset[i-2K,\dots,i+2K]$ and an arbitrary $i$. By Schur's complement formula,
	\ben
		(G_{\J,\J}\wo{\I})^{-1}= H_{\J,\J}-z+H_{\J,\K^c} G\wo{\K}_{\K^c,\K^c} H_{\K^c,\J}\,.
	\ee
The operator norm can be estimated term by term. First,
\ben
	\norma{H_{\J,\J}-z}\leq \sqrt{\sum_{j,j'\in\J} \absa{H_{jj'}}^2}+\absa{z}\,,
\ee
which is $\O{N^\sigma \Phi}+\absa{z}$ with probability $1-c_p N^{-\sigma p}$ by Lemma \ref{lem: ldp}. 
Second,
\ben
	\norma{H_{\J,\K^c} G\wo{\K}_{\K^c,\K^c} H_{\K^c,\J}}\leq \norma{G\wo{\K}_{\K^c,\K^c} } \sum_{j\in\J,k\in\N} \absa{H_{jk}}^2\,.
\ee
Now $\norma{G\wo{\K}_{\K^c,\K^c} } $ is simply bounded by $\eta^{-1}$, while $\sum_{j\in\J,k\in\N} \absa{H_{jk}}^2=\O{1+N^\sigma \Phi}$ with probability at least $1-c_p N^{-\sigma p}$ by Lemma \ref{lem: ldp}. Putting the estimates together and take $\sigma\leq \tau/2$ we have
\ben
	\norma {(G_{\J,\J}\wo{\I})^{-1}}=\O{ \eta^{-1}+\absa{z}}\,,
\ee
with probability at least $1-c_pN^{-\sigma p}$.  The conclusion of the lemma follows from the above estimate and the definition of $\ga$.
\end{proof}
For the readers' convenience, we restate  the global law below:
\begin{theorem}
Let $\D\subset \subset\C^+$.  Let $M$ solve equation \eqref{eq: sce n} and $m$ solve equation \eqref{eq: sce infty}.  Then for arbitrary $\nu>0$, and $p$ large enough, the following estimates hold when $N\geq N_{\D,\nu,p}$.
\ben
	\P{ \sup_{i,j\in\N_N, z\in\D}{\absa{G_{ij}-M_{ij}}}\geq N^\nu \Phi } \leq N^{-\nu p}\,,
\ee
\ben
	\P{\sup_{z\in\D}\absa{\tfrac{1}{N}\tr G- \tr m }\geq N^\nu \Phi}\leq N^{-\nu p}\,.
\ee
\end{theorem}
\begin{proof}
By Lemma  \ref{lem: sce app}  and \ref{lem: gG} and the fact that $G$ is Liptchitz on $\D$ we have
\ben
	G(-z-\Xi(G))=I+R\,,
\ee
where $\sup_{z\in\D} \norma{R}_\infty =\O{N^{2\sigma}\Phi}$ with probability at least $1-c_pN^{-\sigma p+s}$. In other words,
\ben
	G=F(G)+F(G)R\,.
\ee
Here the error term satisfies $\norma{F(G)R}_\infty=\O{\norma{R}_\infty}$ by Lemma \ref{lem: demko}. Now Theorem \ref{thm: stable 1} immediately implies the conclusion, with $\tr m$ replaced by $\tfrac{1}{N}\tr M$.  The proof is concluded by using Theorem \ref{thm: converge} and taking $\nu<\sigma$ small enough.

\end{proof}
%%%%%%%%%%%%%%%%%%%%%%%%%%% local law
%%%%%%%%%%%%%%%%%%%%%%%%%%% local law
%%%%%%%%%%%%%%%%%%%%%%%%%%% local law
\subsection{Proof of the local law}
In order to apply Lemma \ref{lem: sce app}, we need estimates for $\Ga$ and $\ga$.  The estimate for $\ga$ is easy to get when the matrix entries are independent, since the off-diagonals of $G$ are small.  In our case where $G$ has possibly big off-diagonal entries, the estimate for $\ga$ relies on the properties of the solution $M$.
\begin{lemma}\label{lem: ga}
Assume that $M\in\M_N(\beta,\beta^{-1})$ solves equation \eqref{eq: sce n}.  There is an $\e_\beta$ and a $c_\beta$ such that given $\norma{G-M}_\infty\leq \e_\beta$, we have
\ben
	\ga\vee\Ga\leq c_\beta\,.
\ee
\end{lemma}

\begin{proof}
The bound for $\Ga$ is obviously true. In view of definition \eqref{def: ga} of $\ga$, it remains to get a uniform bound for $G_{\I,\I}\wo{\J}$, where $\I\cap\J=\emptyset$, $\I\cup\J\in[i-2K,\dots,i+2K]$, $i\in\N$.  It is sufficient to get a uniform bound for $M_{\I,\I}\wo{\J}:= M_{\I,\I}-M_{\I,\J}(M_{\J,\J})^{-1}M_{\J,\I}$, since $G_{\I,\I}\wo{\J}= M_{\I,\I}\wo{\J}+\O{\e_\beta}$.  Note that $ M\wo{\J}_{\I,\I}=\left(-z-(\Xi(M))\wo{\J}\right)^{-1}_{\I,\I}$, therefore 
\ben
	\inf_{v\in\C^{\I},\norma{v}=1}\absa{ v^*M\wo{\J}_{\I,\I}v}\geq \inf_{v}\frac{\absa{\im v^*\Xi (M)v}}{\norma{(-z-\Xi(M))v}^2}\,.
\ee
which is bounded below by $c\beta^3$, according to Lemma \ref{lem: smooth}.  Therefore, $\norma{(M_{\I,\I}\wo{\J})^{-1}}\leq c\beta^{-3}$ when $\e_\beta$ is small enough.  
\end{proof}

For the readers' convenience, we restate Theorem \ref{thm: local law} below:
\begin{theorem}
Assume that $\xi$ is positive definite in the sense of Definition \ref{def: pd}.   Let $m$ be the solution of equation \eqref{eq: sce infty}.  Fix a bounded domain $\D\subset \C^+$ such that $\im \tr m$ is bounded below by $\om>0$. For arbitrary $\nu\in(0,1]$ let $\D_\nu\wo{N}:=\{z=E+\ii \eta \in \D: \eta>N^{-1+\nu}\}$.
Then for $\sigma$ small enough and $p\geq 100\sigma^{-1}$, the following estimates hold for all $N\geq N_{\om,\sigma,p}$
\ben
	\P{ \sup_{i,j\in\N_N\,,z\in\D_\nu}{\absa{G_{ij}-M_{ij}}}\geq N^\sigma \Phi } \leq N^{-\sigma p}\,,
\ee
\ben
	\P{\sup_{z\in\D_\nu}\absa{\tfrac{1}{N}\tr G- \tr m }\geq N^\sigma \Phi}\leq N^{-\sigma p}\,.
\ee
\end{theorem}
\begin{proof}
Let $\sigma =( \nu\wedge \tau)/20$, so that $N^{5\sigma}\Phi <N^{-\sigma}$ on $\D_\nu$ . All the estimates in this proof hold when $N\ge N_{\om,\sigma,p}$, where $N_{\om,\sigma,p}$ changes from line to line, but only for finite times.  For every $N$, choose a discrete subset $\La\subset \D_\nu$ such that the $N^{-10}$ neighborhood of $\La$ contains $\D_\nu$.  By Lemma \ref{lem: sce app}, we have 
\ben
	\P{\norma{G(-z-\Xi(G))-I}_\infty \geq N^{2\sigma} \Phi\Ga^5\ga^3\,,\exists z\in\La}\leq c_pN^{-\sigma p+s}\,.
\ee
The Green's function $G$ is Liptchitz in $\D_\nu$, with Liptchitz constant $N^2$, since $\absa{\partial_z G_{ij}}=\absa{\sum_k G_{ik}G_{kj}}\leq \eta^{-2}$.  Therefore, the value of $G$ at any point in $\D_\nu$ can be well approximated by the points in $\La$, with error less than $N^{-8}$.  Hence
\ben
	\P{\norma{G(-z-\Xi(G))-I}_\infty \geq N^{3\sigma} \Phi\Ga^5\ga^3\,,\exists z\in\D_\nu}\leq c_pN^{-\sigma p+s}\,.
\ee
By Lemma \ref{lem: ga}, $\norma{G-M}_\infty\leq N^{-\sigma}$ implies that $\Ga\vee \ga=\O{1}$. Therefore we can bound the probability of a smaller set
\ben
	\P{\norma{G(-z-\Xi(G))-I}_\infty\geq N^{4\sigma} \Phi\,,\norma{G-M}_\infty\leq N^{-\sigma}\,,\exists z\in\D_\nu}\leq c_pN^{-\sigma p+s}\,.
\ee
By Theorem \ref{thm: stable bulk}, $\norma{G(-z-\Xi(G))-I}_\infty<N^{4\sigma} \Phi$ implies $\norma{G-M}_\infty\leq N^{5\sigma}\Phi$. Thus the inequality becomes
\ben
	\P{\norma{G-M}_\infty\in (N^{5\sigma}\Phi, N^{-\sigma})\,,\exists z\in\D_\nu}\leq c_pN^{-\sigma p+s}\,.
\ee
In other words,
\ben
	\P{\norma{G-M}_\infty \leq N^{5\sigma}\Phi\,,\forall z\in\D_\nu} +\P{\norma{G-M}_\infty \geq N^{-\sigma}\,,\forall z\in\D_\nu}\geq 1-c_pN^{-\sigma p+s}\,.
\ee
Let $z_0$ be a fixed number in $\D$, the second probability is less than $\P{\norma{G(z_0)-M(z_0)}_\infty \geq N^{-\sigma}}$, which is less than $c_pN^{\sigma p+s}$ by Theorem \ref{thm: global law}.  Therefore,
\ben	
	\P{\norma{G-M}_\infty \leq N^{5\sigma}\Phi\,,\exists z\in\D_\nu} \geq 1-c_pN^{-\sigma p+s}\,.
\ee
The first estimate in the theorem follows from absorbing the constants $c_p$ and $s$ by $N^{-\sigma p/2}$ then replacing $\sigma$ by $5\sigma$ and replacing $p/10$ by $p$.  The second estimate follows from the first estimate and Theorem \ref{thm: converge}.
	
\end{proof}
%%%%%%%%%%%%%%%%%%%%%%%%%%%
\section{Proof of bulk universality}\label{sec: bulk universality}
The strategy is as follows: first, we run an Ornstein-Uhlenbeck process on the matrix entries that has the same correlation structure as $(x_{ij})$ and show that this does not change the local statistics as long as $t\leq N^{-1+\e}$. Second, we prove that the local statistics  at time $t= N^{-1+\e}$ agrees with the local statistics of the GOE. Then we can conclude that the local statistics of the original matrix agrees with the GOE.
\subsection{Correlated Ornstein-Uhlenbeck process}
Let $(B_{ij}(t))_{1\leq i\leq j\leq N}$ be a family of Brownian motions that has the same correlation structure as $(x_{ij})$ does, i.e.,
\ben
	\E {B_{ij}(t)B_{i'j'}(t)}=\xi_{iji'j'}t\,.
\ee
Define $x_{ij}(t)$ through
\ben
	\dd x_{ij}={\dd B_{ij}}-\frac{x_{ij}}{2}\dd t\,.
\ee
Then we define $X(t)$ to be the matrix with upper diagonal part equal to $(x_{ij}(t))_{1\leq i\leq j\leq N}$.  It is easy to check that $X(t)$ has the same correlation structure as $X(0)$ does, and that $(x_{ij}(t))$ satisfy the same moment bounds \eqref{eq: moments}.  Therefore, the local law holds for each $t\geq 0$.  We shall need the following lemma in the sequel:
\begin{lemma}\label{lem: decouple}
	Let $x=(x_1,\dots,x_m)$ be an array of $K$-dependent real centered random variables such that $ \sup_{k} \left(\E {\absa{x_k^3}}\right)^{1/3}\leq \kappa_3$.  Let $f$ be a $C^2$ function on $\R^m$.  Then,
	\ben
		\E{f(x)x_i}=\sum_{k}\E{\pt_{k} f (x)} \E{x_ix_k}+\O{\norma{D^2f}_\infty \kappa_3^3}\,.
	\ee
\end{lemma}
\begin{proof}
	If $f$ is a linear function, then the equality is exact without the error term.  In general, let $\T$ be the set of indices correlated with $i$.   Denote $x\wo{\T}:=(x_k\ind{k\notin\T})$.  By Taylor's expansion,
	\be\label{eq: taylor}
		f(x)=f(x\wo{\T})+\sum_{k\in\T} \pt_{k}f(x\wo{\T}) x_k +\frac{1}{2}\sum_{k,l\in\T}\int_0^1 (1-t) \pt_{kl} f(x\wo{\T}+t(x-x\wo{\T}))x_kx_l\dd t\,.
	\ee
Let $\S\subset \T^c$ be the set of indices correlated with $\T$, and $\U=\S\cup\T$.  Note that $\sum_{k\in\T} \pt_{k}f(x\wo{\T}) x_k =\sum_{k\in\T} \pt_{k}f(x\wo{\U}) x_k +\sum_{k\in \T\,, l\in\S}\int_0^1(1-t)\pt_{kl}f(x\wo{\U}+t(x\wo{\T}-x\wo{\U}))x_kx_l\dd t$.  Denote  
\be\label{eq: def thetax}
	\theta x :=(x_k \ind{k\in\U}\theta_k+x_k\ind{k\notin\U})\,,
\ee
 for $\theta\in[0,1]^\U$,thus \eqref{eq: taylor} can be written
\ben
	f(x)=f(x\wo{\T})+\sum_{k\in\T}\pt_k f(x\wo{\U})x_k+\O{\sum_{k,l\in\U} \sup_{\theta\in[0,1]^\U} \absa{f(\theta x)} \absa{x_kx_l}}\,.
\ee
Now multiply by $x_1$ and take expectation,
\be\label{eq: decouple}
	\E {f(x)x_1}=\sum_{k\in\T}\E{\pt_{k} f (x) }\E{x_1x_k}+ \O{\E {\sum_{k,l\in\U} \sup_{\theta\in[0,1]^\U }\absa{\pt_{kl}f(\theta x)} \absa{x_lx_kx_1}+\absa{x_l\E{x_1x_k}}}}\,.
\ee
Here we have used $\pt_{k}f(x\wo{\U}) =\pt_{k}f(x)-\sum_{l\in\U}\int_0^1(1-t)\pt_{kl}f(x\wo{\U}+t(x-x\wo{\U}))x_l\dd t$.  The conclusion follows.
\end{proof}
Set $H_t:=N^{-1/2}X(t)$.  The above lemma enables us to compare $f(H_t)$ and $f(H_0)$ where $f$ is a $C^3$ function on the matrix space.
\begin{lemma}\label{lem: comparison}
Suppose $f$ is a $C^3$ function on $\C^{N\times N}$.  Then,
\ben
	\E {f(H_t)-f(H_0)}=\O{ t Nq^{-1/2}\E {\norm{D^3 f}_\infty}}\,, \forall t\in\R^+\,.
\ee
\end{lemma}
\begin{proof}
By Ito's formula,
\ben
	\dd \E {f(H_t)}=\sum_{i\leq j}\E{\frac{\pt f}{\pt h_{ij}} h_{ij}}\dd t +\tfrac{1}{N}\sum_{i\leq j\,;i'\leq j'}\E{\frac{\pt f}{\pt h_{ij}\pt h_{i'j'} }}\xi_{iji'j'}\dd t\,.
\ee	
Recall the notation \eqref{eq: def thetax}. Apply \eqref{eq: decouple} to $ \E{\frac{\pt f(H_t)}{\pt h_{ij}}}$ and denote $\T_{ij}=\{(i',j'):\absa{i-i'}\vee\absa{j-j'}\leq 2K\}$. 
\be\label{eq: ibp}\begin{split}
	&\sum_{i,j}\E{\frac{\pt f}{\pt h_{ij}}h_{ij} }=-\tfrac{1}{N}\sum_{i\leq j,i'\leq j'}\E{\frac{\pt f}{\pt h_{ij}\pt h_{i'j'} }}\xi_{iji'j'} \\&+\O{\sum_{i,j}\E{\sum_{(i',j'),(i'',j'')\in\T_{ij}}\sup_{\theta\in[0,1]^{\T_{ij}}}\absa{\frac{\pt^3 f(\theta H)}{\pt_{h_{ij}}\pt_{h_{i'j'}}\pt_{h_{i''j''}}}} \left(\absa{h_{ij}h_{i'j'}h_{i''j''}}+\absa{h_{i''j''}N^{-1}\xi_{iji'j'}}\right)}}\,.
	\end{split}
\ee
Adding the two equations above, the first term on the right hand side of \eqref{eq: ibp} cancels.  The conclusion follows from the observation that $\E{\absa{h_{ij}h_{i'j'}h_{i''j''}}}\leq \mu_3^3 N^{-1}q^{-1/2}$ and $\E{\absa{h_{i''j''}}}\leq N^{-1}$.
\end{proof}

Now we are ready to prove the following Green's function comparison lemma:

\begin{lemma}\label{lem: comparison 2}
Let $\delta>0$ be arbitrary and choose an $\eta$ such that $N^{-1-\delta}\leq \eta \leq N^{-1}$.  For any sequence of positive integers $k_1,\cdots,k_n$ and complex parameters $z_j ^m= E_j^m\pm i\eta, j=1,\dots, k_m$, $m=1,\dots,n$ with an arbitrary choice of the signs and $\rho(E_j)\geq \om$ , we have the following.  Let $G_t(z)=(H_t-z)^{-1}$ be the resolvent and let $f(x_1, x_2, \dots, x_N)$ be a test function such that for any multi-index $\alpha=(\alpha_1, \dots, \alpha_N)$ with $1\leq\absa{\alpha} \leq 3$ and for any positive, sufficiently small $\kappa$, we have 
\ben
	\max\left\{ \absa{\partial^\alpha f(x_1,\dots, x_N)} : \max_j \absa{x_j} \leq N^\kappa \right\} \leq N^{C_0\kappa}
\ee
and
\ben
	\max\left\{ \absa{\partial^\alpha f(x_1,\dots, x_N)} : \max_j \absa{x_j} \leq N^2 \right\} \leq N^{C_0}
\ee
for some constant $C_0$.
Then, for any $\delta$ with $N^{-1-\delta}\leq \eta \leq N^{-1}$ and for any choices of the signs in the imaginary part of $z^m_j$, we have
\ben
	\absa{\E{f\left( \tfrac{1}{N^{k_1}} \tr \Pi_{j=1}^{k_1}G_t(z_j^1),\dots, \tfrac{1}{N^{k_n}} \tr \Pi_{j=1}^{k_n}G_t(z_j^n)\right)} -\E{f(G_t\rightarrow G_0)}}\leq Ct N^{1+c\delta}q^{-1/2}\,,
\ee
where $c$ and $C$ are constants depending on $C_0$.
\end{lemma}

\begin{proof}
We consider only the $n=1$, $k_1=1$ case for simplicity, the general case can be handled likewise.  We want to show that 
\ben
	\absa{\E{f\left( \tfrac{1}{N} \tr G_t(z) \right)} -\E{f\left( \tfrac{1}{N} \tr G_0(z) \right)}}\leq Ct N^{1+c\delta}q^{-1/2}\,.
\ee
In order to apply \eqref{eq: ibp}, we need to bound the derivatives of $\tfrac{1}{N}\tr G_\theta$ with respect to the entries of $H_\theta=N^{-1/2}\theta X$, here $G_\theta=(N^{-1/2}\theta X-z)^{-1}$, $\theta\in[0,1]^{\T_{ij}}$ and $(i,j)\in\N^2$. Note that for $\absa{\alpha}\leq 3$, by direct calculation,
\be\label{eq: off diag bd}
	\absa{\partial^\alpha \tfrac{1}{N} \tr G }\leq \Ga^4\,.
\ee
Note that $\absa{\frac{\pt \Ga}{\pt \eta}}\leq \frac{\Ga}{\eta}$, thus $\Ga(E+i\eta)\leq \Ga(E+iN^{-1+\delta})N^{2\delta}$, where $\Ga(E+iN^{-1+\delta})$ can be bounded using the local law Theorem \ref{thm: local law}.  Therefore, 
\be\label{eq: d tr}
	\absa{\partial^\alpha \tfrac{1}{N} \tr G }\leq CN^{12\delta}\,,
\ee
 on an event with probability $1-N^{-D}$ for some large $D$.  Outside this event, we have a crude bound $\absa{\partial^\alpha \tfrac{1}{N} \tr G }\leq C N^{8}$. 

It is not hard to show that Theorem \ref{thm: local law} holds for $G_\theta$ uniformly for $\theta\in\T_{ij}$ and all $(i,j)\in\N^2$ by a continuity argument.  Therefore the estimate \eqref{eq: d tr} holds uniformly for $\theta\in\T_{ij}$ and $(i,j)\in\N^2$ with some larger $D>0$.

Finally, we apply Lemma \ref{lem: comparison} (or rather, the estimate \eqref{eq: ibp}) to complete the proof. 
\end{proof}

Lemma \ref{lem: comparison 2} enables us to approximate test functions on the microscopic scale by a standard argument. We have the following comparison theorem.  For the proof we refer the readers to Theorem 6.4 in \cite{Erdos2012}\,.  
\begin{theorem}\label{thm: comparison}
Suppose the assumptions of Theorem \ref{thm: local law} hold. Let $t=N^{-1+\e}$ for some $\e>0$ small enough. Let $p^{(k)}_{t,N}$ be the $k$-point correlation function of the eigenvalues of $H_t$.  Assume $E\in\R$ such that $\rho$ has a positive finite density on a neighborhood of $E$.  Then, for any compactly support continuous test function $O:\R^k\to\R$ we have for some $c=c(E)>0$,
\ben
	\int_{\R^k}\dd \alpha_1\dots\dd\alpha_k \left(p_{0,N}^{(k)}-p_{t,N}^{(k)}\right)(E+\alpha_1/N,\dots,E+\alpha_k/N)=\O{N^{-c}}\,.
\ee
\end{theorem}
It remains to show that the local statistics of $H_t$ in the bulk agrees with that of the GOE.  A important observation is that $\dd B$ can be decomposed
\ben
	\dd B_{ij}=\dd B^1_{ij}+\dd B_{ij}^2\,,
\ee
such that $B^1$ and $B^2$ are independent and $\dd B_{ij}^2$ is a GOE scaled by a small constant.  Therefore, when $t\geq N^{-1+\e}$, we can write $x_{ij}(t)=\tilde x_{ij}(t) + w_{ij}$, where $(w_{ij})$ are i.i.d. Gaussian with variance $N^{-1+\e}$  and $(\tilde x_{ij}(t))$ is a family of random variables that have a positive definite correlation structure $\tilde \xi$ which is positive definite in the sense of Definition \ref{def: pd} with lower bound $c_0-o(1)$.  It is tedious but easy to show that the local law holds for $\tilde X(t)=(\tilde x_{ij}(t))$. Denote $\tilde H_t:=N^{-1/2}\tilde X(t)$ and $\tilde G:=(\tilde H_t-z)^{-1}$.  The local law implies that $\im \tfrac{1}{N}\tr \tilde G(z)$ is bounded above and bounded below for $z=E+iN^{-1+\e}$ where $E$ is in the bulk of the limiting spectrum.

\subsection{Comparison with GOE}
Theorem \ref{thm: bulk universality} is a corollary of Theorem \ref{thm: comparison} above and Theorem 2.4 in \cite{Landon2015a}.  Here we remark that the result in \cite{Schnelli15} also applies to our case.
\begin{proof}[Proof of Theorem \ref{thm: bulk universality}]
Let $\D=\{\zeta: \re\zeta\in[E-a,E+a], \im\zeta \in[0,1]\}$ for some small constant $a>0$.  Therefore, $\im \tr m(z)$ is bounded below on $\D$ when $a$ is small enough.  Set $t=N^{-1+\e}$ with $\e>0$ small enough.  As we have seen in the last subsection, with probability $1-N^{-\sigma p}$, $\im \tfrac{1}{N}\tr \tilde G$ is bounded above and bounded below on 
\ben
	\D_\nu\wo{N}:=\{\zeta \in \D: \im\zeta>N^{-1+\nu}\}\,.
\ee
Consequently, $\tilde H(t)$ satisfies the regularity condition in Theorem 2.4 of \cite{Landon2015a}. Therefore, the $k$-point correlation function of $H_t$ near $E$ converges to that of the GOE in the sense of Theorem 2.4 of \cite{Landon2015a}.  The conclusion of Theorem \ref{thm: bulk universality} follows in view of Theorem \ref{thm: comparison}.
\end{proof}

\textbf{Acknowledgements.} Part of the work was done during the author's stay at NCTS (National Center of Theoretical Sciences), Taipei. The author would like to thank NCTS for their hospitality and help. The author also would like to thank Benjamin Landon for his help on proofreading the article and improving presentations in many parts of the article.  The author would like to thank Prof. H.-T. Yau for helpful comments on this article.

\bibliographystyle{abbrv}
%\bibliography{/Users/ziliangche/Desktop/tex/randommatrix.bib}
\bibliography{randommatrix.bib}
\end{document}